\documentclass[final,1p,times]{elsarticle}
\usepackage{amssymb,amsthm,amsmath, amsfonts}
\usepackage{colortbl}
\usepackage{float}
\usepackage{booktabs}
\usepackage{algorithm}
\usepackage{algpseudocode}
\usepackage{xcolor}
\usepackage{hyperref}
\hypersetup{
    colorlinks=true,
    urlcolor=blue,     % or black
    linkcolor=black,
    citecolor=black
}

\newtheorem{theorem}{Theorem}[section]

\newtheorem{lemma}{Lemma}[section]
\newtheorem{assumption}{Assumption}

\DeclareMathOperator{\supp}{supp}

\journal{Computer Methods in Applied Mechanics and Engineering}

\begin{document}
\maketitle
\begin{frontmatter}


% \title{\tnoteref{label1}}

\title{Variational Data-Consistent Assimilation} %% Article title

\author[label1]{Rylan Spence}
\author[label2]{Troy Butler}
\author[label1]{Clint Dawson}
\affiliation[label1]{organization={Oden Institute for Computational Science and Engineering, University of Texas},
            addressline={201 E 24th St},
            city={Austin},
            postcode={78712},
            state={TX},
            country={U.S.}}

\affiliation[label2]{organization={Department of Mathematical and Statistical Sciences, University of Colorado Denver},
            addressline={4000 1201 Larimer St},
            city={Denver},
            postcode={80202},
            state={CO},
            country={U.S.}}

%% Abstract
\begin{abstract}
This work introduces a new class of four-dimensional variational data assimilation (4D-Var) methods grounded in data-consistent inversion (DCI) theory. 
The methods extend classical 4D-Var by incorporating a predictability-aware regularization term. 
The first method formulated is referred to as Data-Consistent 4D-Var (DC-4DVar), which is then enhanced using a Weighted Mean Error (WME) quantity-of-interest map to construct the DC-WME 4D-Var method. 
While the DC and DC-WME cost functions both involve a predictability-aware regularization term, the DC-WME function includes a modification to the model-data misfit, thereby improving estimation accuracy, robustness, and theoretical consistency in nonlinear and partially observed dynamical systems. 
Proofs are provided that establish the existence and uniqueness of the minimizer and analyze how a predictability assumption that is common within the DCI framework helps to promote solution stability.
Numerical experiments are presented on benchmark dynamical systems (Lorenz-63 and Lorenz-96) as well as for the shallow water equations (SWE). 
In the benchmark dynamical systems, the DC-WME 4D-Var formulation is shown to consistently outperform standard 4D-Var in reducing both error and bias while maintaining robustness under high observation noise and short assimilation windows. 
Despite introducing modest computational overhead, DC-WME 4D-Var delivers improvements in estimation performance and forecast skill, demonstrating its potential practicality and scalability for high-dimensional data assimilation problems.
\end{abstract}


% \begin{highlights}
% \end{highlights}

%% Keywords

\end{frontmatter}



\section{Introduction}\label{sec:intro}

\subsection{Variational Data-Assimilation Background}
Variational data assimilation refers to a class of data assimilation algorithms in which the estimated fields are obtained by minimizing a scalar objective function (i.e., a cost function) that quantifies the mismatch between model outputs and available observations. 
In the late 1960s, variational data assimilation emerged as a more unified and dynamically consistent alternative to the early numerical weather prediction (NWP) systems that relied on objective analysis methods such as successive corrections and optimal interpolation to generate initial conditions.  
\cite{Sasaki1970} introduced a formulation, distinct from earlier statistical approaches, based on deterministic optimal control, casting the assimilation problem as a constrained minimization problem with the forecast model as an equality constraint.

By the 1990s, operational centers had adopted three-dimensional variational assimilation (3D-Var), which minimized a cost function incorporating background and observational constraints that enabled the direct assimilation of diverse data types with statistically weighted errors.
Furthermore, 3D-Var presented clear advantages for optimization-based approaches, which set the stage for time-continuous extensions such as four-dimensional variational assimilation (4D-Var). 
4D-Var generalizes the 3D-Var framework by incorporating observations over a finite time window, rather than restricting assimilation to a single analysis time. 
\cite{Dimet1986} and \cite{Lewis1985} independently introduced the 4D-Var concept, which seeks the optimal initial state whose model trajectory best fits all observations within the assimilation window. 
The associated cost function typically consists of a background term penalizing deviations from a prior estimate and an observation term penalizing misfits with time-distributed data referred to as either the model-data misfit or more simply as the data-misfit. 
Classical 4D-Var treats the forecast model as a strong constraint by propagating the state forward exactly from the initial time, consistent with the perfect model assumption first proposed by Sasaki. 
The 4D-Var framework offers a key advantage over 3D-Var: it assimilates data at their actual time of availability by using the forecast model to interpolate between observation times, which enables the system to dynamically propagate observational influence, enforce flow-dependent multivariate structure, and generate physically balanced analyses. 

Unlike the standard Kalman filter and other sequential assimilation methods, 4D-Var propagates observational information both forward and backward in time. 
These properties make 4D-Var particularly effective in handling high-frequency, asynchronous, or sparse observations, such as those from satellites and radar. 
Solving the 4D-Var problem generally requires tools from optimal control theory such as adjoint models utilized to compute gradients of the cost function. 
The forecast equations are enforced using Lagrange multipliers, resulting in an Euler–Lagrange system comprising the forward and adjoint models. 
In each optimization iteration, the nonlinear model is integrated forward to compute the data-misfits, followed by a backward integration of the adjoint model to evaluate the gradient with respect to the initial state. 
The adjoint model, defined as the transpose of the linearized forecast model, efficiently computes sensitivities by propagating them backward in time. 
This technique allows for scalable gradient evaluations at a cost comparable to a few model runs, regardless of observation count~\citep{Courtier1987}. 

Developing tangent-linear and adjoint models for a full NWP system remains challenging. 
Every dynamical and physical process in the forecast model must be linearized and differentiated. 
Early 4D-Var implementations relied on simplified physics to maintain numerical stability, but more recent systems include advanced linearized physics schemes, improving accuracy, especially in tropical and convectively active regions. 
Because strong-constraint 4D-Var is computationally demanding, operational centers adopt incremental 4D-Var~\citep{Courtier1994}, which linearizes the problem around a background trajectory and solves a sequence of lower-dimensional subproblems. 
Another refinement is weak-constraint 4D-Var, which relaxes the perfect-model assumption by allowing an explicit model error terms in the control vector~\citep{Lewis1985}. 
This extension is critical for longer assimilation windows or models with structural uncertainty. 
4D-Var builds upon the deterministic, optimal control foundations established by Sasaki and extends the variational approach into the time dimension. 
Supported by adjoint-based gradient computation and strengthened by successive algorithmic refinements, 4D-Var continues to serve as a cornerstone of high-resolution, time-continuous data assimilation in modern numerical weather prediction~\citep{Gustafsson2012,Rabier2000}.

\subsection{Data-Consistent Inversion (DCI) Background}
Inverse problems have long been concerned with recovering parameters that best explain observed data, often through optimization techniques or Bayesian inference. 
However, traditional approaches typically aim to identify a single optimal parameter or a posterior distribution that quantifies uncertainties for a particular estimate of a parameter based on a likelihood model. 
In contrast, data-consistent inversion (DCI) is based on a reconceptualization of the inverse problem in a measure-theoretic context. 
Instead of identifying point estimates or conditional densities, the inferential target of DCI is a probability measure over the parameter space whose subsequent push-forward through the forward model reproduces the observed distribution on the output space. 
In other words, the DCI solution is a pullback of the observed probability measure. 
The roots of this framework are found in~\cite{Butler2012,Breidt2011,Butler2014}, where ``inverse sensitivity'' methods are utilized to quantify how uncertainties in model inputs translate to uncertainties in outputs. 
In this formulation, the solution to the inverse problem is defined as any measure on the input space whose pushforward matches the law of the observed data. 
The early work of~\cite{Butler2013} analyzed the consistency and convergence of foundational algorithms that explicitly approximated such input measures based on the intersection of specified parameter events with ``generalized contour events'' defined by the inverse image of output events of known, or approximated, probabilities.
The formulation proved particularly useful in problems with limited data or non-invertible forward maps, where traditional methods often fail.

This type of stochastic inverse problem is often underdetermined and may admit infinitely many solutions, necessitating regularization or selection criteria to obtain physically meaningful results \citep{Bingham2024}. 
To resolve this non-uniqueness,~\cite{Butler2018} proposed a formulation that introduces a Bayesian-style update to an ``initial'' distribution. 
This update ensures that the revised input distribution both honors prior information while simultaneously having the push-forward match the data distribution. 
The result is a unique (up to the choice of the initial distribution) data-consistent distribution referred to simply as ``the update.'' 
This forms the foundation of the modern DCI framework. 
Applications have followed in diverse areas.
\cite{Butler2020} employed DCI for optimal experimental design, identifying experimental conditions that minimize uncertainty in the pushforward. 
\cite{Tran2021} showed that DCI integrates naturally with machine learning surrogates, making it feasible to solve high-dimensional inverse problems where likelihoods are intractable or unavailable.

Building on this foundation,~\cite{Pilosov2023} introduced the concept of Maximal Updated Densities (MUD) as a non-iterative update rule that identifies and quantifies uncertainties in point estimates of input values. 
Unlike Bayesian updates, a MUD estimate requires no likelihood function and is well-suited for expensive or implicit forward models. 
This methodology was further extended to the time-dependent setting via Sequential MUD (SMUD) estimation~\citep{del-Castillo-Negrete2024}, which updates the input density incrementally as new data arrive, enabling real-time assimilation while preserving data-consistency of the push-forward. 
These recent developments provide computationally efficient and theoretically grounded alternatives to traditional Bayesian inference that are particularly attractive in modern inverse problems where full probabilistic characterizations are needed and likelihoods are either inaccessible or prohibitively costly to compute.

DCI has also begun to influence applied fields such as hydrology and meteorology. 
For example,~\cite{Graham2015} inferred distributions over Manning’s roughness coefficients in a hydrodynamic model to match uncertainty in observed water height distributions. 
In meteorology applications, \cite{Fleury2025} utilize DCI to derive parameter distributions that improve the statistical consistency (bias reduction) of ensemble forecasts in atmospheric modeling. 
Although not yet widely used in operational NWP, DCI offers a new path forward by assimilating full output distributions that reflect the uncertainty in observed data, improving uncertainty quantification in forecast systems.

\subsection{Benefits to Merging 4D-Var and DCI}
DCI is an inversion method that addresses aleatoric (i.e., irreducible) uncertainty through probability densities \citep{Butler2018} and epistemic (i.e., reducible) uncertainty through Maximal Updated Densities \citep{Pilosov2023}. 
As of this writing, limited work has been done to reformulate the DCI framework into a deterministic optimization problem. 
In the context of data assimilation, such a reformulation corresponds to a shift from filtering and smoothing algorithms to a variational prospective.
This work focuses specifically on the development and analysis of such a DCI reformulation as an alternative to traditional 4D-Var that can provide specific benefits in certain scenarios. 
We provide two reformulations referred to as Data-Consistent 4D-Var (DC 4D-Var) and Data-Consistent Weighted Mean Error 4D-Var (DC-WME 4D-Var).
The DC-WME 4D-Var proves to have many benefits due to its incorporation of predictability-aware regularization and quantity-of-interest (QoI)-focused optimization. 
Below, we give a high-level summary of the benefits of DC-WME 4D-Var over traditional 4D-Var. 

As we increase the number of state variables, standard 4D-Var suffers from underdetermined inference in sparse observation settings. 
When only partial observations are available, such as observing every other component of the state, the number of unobserved components grows with system size. 
This growth introduces greater uncertainty in unobserved directions. 
Because standard 4D-Var does not explicitly model or propagate this uncertainty, its estimates become increasingly biased and less accurate in those under-observed dimensions. 
In DC-WME 4D-Var, the predictability-aware regularization introduces model-derived information about which directions in state space exhibit greater or lesser stability or uncertainty over time. 
This structure allows the assimilation scheme to penalize errors more strongly in predictable directions while allowing more flexibility in uncertain or weakly constrained modes. 

Another major challenge in 4D-Var arises from the poor conditioning of the optimization problem. 
As the system dimension increases, the cost function becomes more ill-conditioned. 
The background and observation covariance matrices must capture increasingly complex error structures, which we often find difficult to specify or approximate in practice. 
This poor conditioning can slow the convergence of the optimization routine or lead it to converge to suboptimal solutions.
The QoI-focused optimization of DC-WME 4D-Var shifts the objective from uniformly fitting the entire state trajectory to prioritizing accuracy in components of the solution most relevant to the forecast or application, such as specific observations or downstream functions of the state. 
Together, these modifications produce a more balanced and targeted assimilation strategy that accounts for both model dynamics and observational coverage. 
As a result, DC-WME 4D-Var maintains robust performance even as the system dimensionality increases, mitigating the degradation typically observed in high-dimensional settings.

Another known issue in standard 4D-Var is the high sensitivity to mismatches between model and observation error statistics. 
In large systems, even small model errors or incorrect assumptions about observational noise can have outsized effects, especially in state directions that the data weakly constrain. 
Because standard 4D-Var lacks a mechanism for adaptively weighting data based on informativeness or predictability, it becomes prone to assimilating low-quality or misleading observations. 
The cost function imposes uniform penalties using fixed covariance matrices, regardless of the dynamical importance or uncertainty associated with different state directions. 
As a result, the algorithm may overfit well observed directions while underfitting critical but unobserved modes, which is a particularly severe issue in high dimensional, chaotic systems such as Lorenz-96. 
The DC-WME 4D-Var further explicitly accounts for model uncertainty. 
When the model is subject to significant epistemic error, probabilistic calibration ensures that its predictions remain consistent with observed outcomes, improving reliability even in partially observed regimes. 
By incorporating a predictability constraint, the DC-WME 4D-Var method identifies which aspects of the model output are both informative and stable over time, allowing the assimilation algorithm to focus on the most meaningful components of the solution. 
This principled filtering of state information enhances robustness and interpretability, particularly in high-dimensional systems where traditional methods struggle.

\subsection{Contributions}

This work presents several key contributions to the development and application of DCI-based 4D-Var. 
It begins by introducing Data-Consistent Inversion (DCI) to the field of data assimilation for state estimation, bridging concepts from uncertainty quantification and statistical inversion with dynamic modeling. 
Building on this foundation, the DCI framework is reformulated from its original probabilistic setting into a variational one, enabling seamless integration into 4D-Var methodologies. 
Two novel DC-based 4D-Var algorithms are then introduced to enhance the assimilation of observational data culminating in the preferred DC-WME 4D-Var method. 
The adjoint method is developed and implemented for multiple QoI maps within the DC-WME 4D-Var framework, allowing for scalable and efficient gradient computation. 
Theoretical analysis establishes the existence and uniqueness of minimizers for the new methods, ensuring the well-posedness of the optimization problem. 
These methods are applied to a range of nonlinear weather models, demonstrating their accuracy and robustness. 
Lastly, the framework is extended to a coastal application, highlighting its versatility and effectiveness in real-world geophysical systems.

%\\\\\\\\\\\\\\\\\\\\\\\\\\\\\\\\\\\\\\\\\\\\\\\\\\\\\\\\\\\\\\\\\\\\\\\\\\\\\\\\\\\\\\\\\\\\\\\\\\\\\\\\\\\\\\\\\\\\\\\\\\\\\\\\\\\\\\\\\\\\\\\\\\\\\\\\\\\\\\\\\\\\\\\\\\\
\section{Background on DCI-Based Point Estimation}\label{sec:Background}

This section provides the necessary background on point estimation within the DCI framework that is required to develop the DC and DC-WME 4D-Var methods.
We first introduce some notation.
Let the DCI solution be denoted by 
\begin{equation}
        \pi_{\mathrm{up}} \left(\mathbf{z}_{0}\right) := \pi_{\mathrm{init}} \left(\mathbf{z}_{0}\right) \frac{\pi_{\mathrm{obs}}\left( Q\left(\mathbf{z}_{0}\right) \right) }{\pi_{\mathrm{pred}} \left( Q \left(\mathbf{z}_{0}\right) \right)}, \label{eq:dci_problem1}
\end{equation}
where $\mathbf{z}_0$ denotes an initial state, $\pi_{\mathrm{init}}$ quantifies the initial (background) uncertainty in that state, $Q$ is the QoI map from states to observations, $\pi_{\mathrm{pred}}$ is the push-forward of $\pi_{\mathrm{init}}$ through $Q$, and $\pi_{\mathrm{obs}}$ quantifies uncertainties in the observed QoI. 
We refer the interested reader to Appendices~\ref{app:consistency}-\ref{app:MUD} and the references therein for a more general overview of DCI.

%\\\\\\\\\\\\\\\\\\\\\\\\\\\\\\\\\\\\\\\\\\\\\\\\\\\\\\\\\\\\\\\\\\\\\\\\\\\\\\\\\\\\\\\\\\\\\\\\\\\\\\\\\\\\\\\\\\\\\\\\\\\\\\\\\\\\\\\\\\\\\\\\\\\\\\\\\\\\\\\\\\\\\\\\\\\
%\\\\\\\\\\\\\\\\\\\\\\\\\\\\\\\\\\\\\\\\\\\\\\\\\\\\\\\\\\\\\\\\\\\\\\\\\\\\\\\\\\\\\\\\\\\\\\\\\\\\\\\\\\\\\\\\\\\\\\\\\\\\\\\\\\\\\\\\\\\\\\\\\\\\\\\\\\\\\\\\\\\\\\\\\\\

\subsection{Defining the MUD State Estimation}\label{subsec:mud_state_estimate}

We first assume that the initial, predicted, and observed distributions are given by the Gaussian distributions, $\mathcal{N}(\mathbf{z}_0^b, \mathbf{B})$, $\mathcal{N}(Q(\mathbf{z}_0^b), \mathbf{L})$, and $\mathcal{N}(\mathbf{y}, \mathbf{R})$, respectively. 
Similar to linear Gaussian inversion problems within the Bayesian framework, there is a connection between the data-consistent problem and its solution, $\pi_{\mathrm{up}}$, and more traditional variational inverse problem approaches. 
Consider the following optimization problem
\begin{equation} 
    z^{\mathrm{MUD}}:=  \underset{\mathbf{z}_{0}\in \mathcal{Z}}{\arg\max} \  \pi_{\mathrm{up}} \left(\mathbf{z}_{0} \right)
    = \underset{\mathbf{z}_{0} \in \mathcal{Z}}{\arg \max} \  \exp \left( -  \mathcal{J}\left(\mathbf{z}_{0}\right)   \right), \label{eq:mud_exp}
\end{equation}
where we define the data-consistent (DC) objective function $\mathcal{J}:\mathcal{Z} \to \mathbb{R}\cup \{\infty\}$ as 
\begin{equation}
\begin{split}
    \mathcal{J}\left(\mathbf{z}_{0}\right)  
    = \frac{1}{2}\| \mathbf{z}_{0} - \mathbf{z}_0^b \|_{\mathbf{B}^{-1}} 
    &+ \frac{1}{2}\| Q\left(\mathbf{z}_{0} \right) - \mathbf{y}\|_{\mathbf{R}^{-1}}  \\
    &- \frac{1}{2}\| Q\left(\mathbf{z}_{0} \right) - Q\left(\mathbf{z}_0^b\right) \|_{\mathbf{L}^{-1}} . \label{eq:dc_obj}
\end{split}
\end{equation}

%\\\\\\\\\\\\\\\\\\\\\\\\\\\\\\\\\\\\\\\\\\\\\\\\\\\\\\\\\\\\\\\\\\\\\\\\\\\\\\\\\\\\\\\\\\\\\\\\\\\\\\\\\\\\\\\\\\\\\\\\\\\\\\\\\\\\\\\\\\\\\\\\\\\\\\\\\\\\\\\\\\\\\\\\\\\
When $Q$ is a linear operator, rewriting the DC objective function as
\begin{equation}
      \mathcal{J}\left(\mathbf{z}_{0}\right)  = \frac{1}{2}\| \mathbf{z}_{0} - \mathbf{z}_0^b \|_{\Psi^{-1}} +  \frac{1}{2}\|Q\left(\mathbf{z}_{0} \right) - \mathbf{y} \|_{\mathbf{R}^{-1}},   \label{eq:dc_obj_psi}
\end{equation}
where $\Psi^{-1} = \mathbf{B}^{-1} - Q^\top\mathbf{L}^{-1} Q$, represents this as the cost function associated with a modified Bayesian inverse problem with a prior distribution given by $ \mathcal{N} \left(\mathbf{z}_0^b,\Psi\right) $. 
In other words, we can view both the MUD and MAP points as solutions to distinct Bayesian inverse problems. 
%After defining the data-consistent objective function, we now examine one of its key properties. 
Furthermore, in the context of data assimilation,~\eqref{eq:dc_obj_psi} serves as a modified version of the function used in the 3D-Var to compute the optimal state estimate for a given assimilation window of data.

%\\\\\\\\\\\\\\\\\\\\\\\\\\\\\\\\\\\\\\\\\\\\\\\\\\\\\\\\\\\\\\\\\\\\\\\\\\\\\\\\\\\\\\\\\\\\\\\\\\\\\\\\\\\\\\\\\\\\\\\\\\\\\\\\\\\\\\\\\\\\\\\\\\\\\\\\\\\\\\\\\\\\\\\\\\\

\subsection{Existence and Uniqueness of MUD Point}\label{subsec:existence_unique_mud}

Strict convexity of a cost function ensures that if a solution exists to the associated deterministic optimization problem, the extrema is a strict, and therefore unique, minimizer or maximizer. 
In Appendix~\ref{append:convexity} we prove that the DC objective functions utilized in this work are strictly convex. 
Of particular note that in the convexity theorem is that the Hessian depends only on the background and ratio precision matrices while not being dependent on the state $\mathbf{z}_{0}$. 
This result highlights a key feature of the DCI-based solution, which is the weighting of the difference in uncertainty between what is being observed from an underlying system and model predictions. 
Furthermore, akin to conventional data assimilation theory, we later demonstrate that the inverse of the Hessian serves as the updated covariance. While a similar result was derived via linear algebra in~\cite{Pilosov2023} in terms of system parameters, this work provides the proof in the variational setting along with an alternative perspective on the original result based on the state of the system.

%\\\\\\\\\\\\\\\\\\\\\\\\\\\\\\\\\\\\\\\\\\\\\\\\\\\\\\\\\\\\\\\\\\\\\\\\\\\\\\\\\\\\\\\\\\\\\\\\\\\\\\\\\\\\\\\\\\\\\\\\\\\\\\\\\\\\\\\\\\\\\\\\\\\\\\\\\\\\\\\\\\\\\\\\\\\

\subsection{Ratio Precision Matrix and the Predictability Assumption}\label{subsec:ratio_matrices}

Note that the updated density can be written as
\begin{equation}
    \pi_{\mathrm{up}}(\mathbf{z}_0) = \pi_{\mathrm{init}}(\mathbf{z}_0)r(\mathbf{z}_0), \ r(\mathbf{z}_0):= \frac{\pi_{\mathrm{obs}}\left( Q\left(\mathbf{z}_{0}\right) \right) }{\pi_{\mathrm{pred}} \left( Q \left(\mathbf{z}_{0}\right) \right)}.
\end{equation}

Both~\cite{Butler2018} and~\cite{Pilosov2023} utilize the expected value of $r$ as a metric to understand the validity of results in DCI problems. 
Specifically, the sample average (with respect to samples drawn from the initial distribution) of $r$ should be approximately unity.
This is guaranteed if sufficient approximations to the densities are utilized and the predictability assumption described in Theorem~\ref{thm:sip_exist_unique} of
Appendix~\ref{app:DCI_solution} is satisfied.
At a conceptual level, the predictability assumption simply states that we must be able to predict what we are likely to observe. 

To understand the predictability assumption in the linear Gaussian case, we first define the ratio precision matrix as $\mathbf{W}^{-1} := \mathbf{R}^{-1} - \mathbf{L}^{-1}$.
A conservative approach to ensure that the predictability assumption holds is to have the minimum singular value of $\mathbf{L}$ be greater than the maximum singular value of $\mathbf{R}$, i.e., $\sigma_{\mathrm{min}}\left( \mathbf{L}\right) > \sigma_{\mathrm{max}}\left( \mathbf{R}\right)$, which implies $\sigma_{\mathrm{min}}\left( \mathbf{R}^{-1}\right) > \sigma_{\mathrm{max}}\left( \mathbf{L}^{-1}\right)$. 
At a conceptual level, this ensures that the slowest asymptotic decay of the observed distribution in a particular direction is still faster than the asymptotic decay of the predicted distribution in any direction. 
Subsequently, $\mathbf{W}^{-1}$ is symmetric positive definite (SPD) and thus invertible with $\mathbf{W}$ defining the ratio covariance. 
In~\cite{Pilosov2023}, it is shown that if $Q$ is defined by a weighted mean error map for linear observations, then after a sufficient amount of observable data are collected, it is guaranteed that $\sigma_{\mathrm{min}}\left( \mathbf{R}^{-1}\right) > \sigma_{\mathrm{max}}\left( \mathbf{L}^{-1}\right)$. 
It is worth noting that, in practice, this is a stronger condition than is generally necessary to ensure that $\mathbf{W}^{-1}$ is SPD.
We demonstrate this in the numerical results by relaxing this condition.

%\\\\\\\\\\\\\\\\\\\\\\\\\\\\\\\\\\\\\\\\\\\\\\\\\\\\\\\\\\\\\\\\\\\\\\\\\\\\\\\\\\\\\\\\\\\\\\\\\\\\\\\\\\\\\\\\\\\\\\\\\\\\\\\\\\\\\\\\\\\\\\\\\\\\\\\\\\\\\\\\\\\\\\\\\\\
\section{DCI-based 4D-Var Methodology}\label{sec:Methodology}
4D-Var data assimilation seeks an optimal model trajectory that best fits all available observations within a short analysis window by adjusting the initial conditions of a numerical forecast model~\citep{Talagrand1987}. 
The method is formulated as the minimization of a cost function defined over this window, where each term is weighted according to the estimated error statistics of both the forecast model and the observations. 
In its ``strong constraint'' formulation, 4D-Var assumes that the model dynamics are perfect, meaning they are free of error. 
This work adopts the strong constraint formulation of 4D-Var. 
Under this assumption, at any time $ t_k $, denote by $ \mathbf{z}_k$ the true model state that is entirely determined by propagating the initial state $\mathbf{z}_0$ through the model dynamics, i.e., 
\begin{equation}
    \mathbf{z}_k = \mathcal{M}_{k:0}(\mathbf{z}_0), 
\end{equation}
where $ \mathcal{M}_{k:0} $ denotes the composition of the model maps from time $ t_0 $ to time $ t_k $. For example, assuming a sequential model structure:
\begin{equation}
\begin{split}
    \mathbf{z}_1 &= \mathcal{M}_{0}(\mathbf{z}_0), \\
    \mathbf{z}_2 &= \mathcal{M}_{1}(\mathbf{z}_1) 
                 = \mathcal{M}_{1} \circ \mathcal{M}_{0}(\mathbf{z}_0), \\
                 &\;\;\vdots \\
    \mathbf{z}_k &= \mathcal{M}_{k-1} \circ \cdots \circ \mathcal{M}_{0}(\mathbf{z}_0) 
                 = \mathcal{M}_{k:0}(\mathbf{z}_0). \label{eq:strong_constraint}
\end{split}
\end{equation}
In practice, we start with a ``background'' estimate of the initial state, denoted by $\mathbf{z}_0^b$ that is associated with the error covariance matrix $ \mathbf{B} $. 
The observational error covariance is given by the matrix $ \mathbf{R} $. 

Thus, in strong-constraint 4D-Var, the entire trajectory $\{\mathbf{z}_k\}_{k=0}^N$ is a deterministic function of the initial state $ \mathbf{z}_0 $, and the optimization is performed only over $\mathbf{z}_0$. 
The 4D-Var analysis, denoted by $\mathbf{z}^{\mathrm{a}}$, is obtained by minimizing the cost function
\begin{equation}
\begin{split}
    \mathcal{J}(\mathbf{z}_{0}) 
    &= \left(\mathbf{z}_{0}-\mathbf{z}_0^b\right)^{\top} 
       \mathbf{B}^{-1}\left(\mathbf{z}_{0}-\mathbf{z}_0^b\right) \\
    &\quad + \sum_{k=0}^{N} 
       \left(\mathcal{H}_k(\mathbf{z}_{k})-\mathbf{y}_k\right)^{\top} 
       \mathbf{R}^{-1}\left(\mathcal{H}_{k}(\mathbf{z}_{k})-\mathbf{y}_k\right). \label{eq:4dvar_cost}
\end{split}
\end{equation}
which quantifies the discrepancy between the model trajectory and the observations at discrete time points $ t_k $ for $ k = 0, \dots, N $.
Here, $\mathcal{H}_k$ denotes the observation operator that maps the state, $\mathbf{z}_k$, to model predicted data that are subsequently compared to (noisy) observed data, $\mathbf{y}_k$, at time $t_k$.
Through this minimization, 4D-Var estimates the optimal state trajectory, ensuring it remains close to observed values within the assimilation window while also aligning with the background state at the window’s start.
Mathematically, the analysis $ \mathbf{z}_{0}^a $ at time $ t_0 $ corresponds to the initial state $ \mathbf{z}_{0} $ that minimizes the cost function~\eqref{eq:4dvar_cost} subject to the model constraint of~\eqref{eq:strong_constraint}.

Computing the optimal state at every time step is computationally prohibitive. However, by reducing the control variable, the method modifies only the initial conditions, enforcing strict adherence to the model equations throughout the trajectory~\citep{Dimet1986}. 
The cost function is minimized by adjusting the model’s initial condition $ \mathbf{z}_{0} $ at $ t_{0} $. 
This process requires integrating the full nonlinear model forward to the end of the assimilation window to compute the cost function, followed by a backward integration of the adjoint model to determine the cost function’s gradient with respect to $ \mathbf{z}_{0} $. 
An iterative descent algorithm then uses this gradient information. 
The final state at the end of the assimilation window then serves as the initial guess for the next assimilation cycle, i.e., $\mathbf{z}_N$ serves as the new initial state.

%\\\\\\\\\\\\\\\\\\\\\\\\\\\\\\\\\\\\\\\\\\\\\\\\\\\\\\\\\\\\\\\\\\\\\\\\\\\\\\\\\\\\\\\\\\\\\\\\\\\\\\\\\\\\\\\\\\\\\\\\\\\\\\\\\\\\\\\\\\\\\\\\\\\\\\\\\\\\\\\\\\\\\\\\\\\

\subsection{Data-Consistent 4D-Var (DC 4D-Var)}
We now seek to merge the DCI and 4D-Var frameworks by connecting the 4D-Var cost function of \eqref{eq:4dvar_cost} to the DC cost function of \eqref{eq:dc_obj}.
To do this, we first identify the $k$th QoI map as $Q_k = \mathcal{H}_k \circ\mathcal{M}_{k:0}$.
Then, we introduce the following modified penalty term defining the predictability-aware regularization term:

\begin{equation}
    \frac{1}{2} \sum_{k=0}^{N} \left( Q_k \mathbf{z}_0 - Q_k \mathbf{z}_0^b \right)^\top \mathbf{L}_k^{-1} \left( Q_k \mathbf{z}_0 - Q_k \mathbf{z}_0^b \right) \label{eq:predict_norm}
\end{equation}
The full modified DC cost function becomes:
\begin{equation}
\begin{split}
    \mathcal{J}_{\mathrm{DC}}(\mathbf{z}_0) 
    :=& \;\frac{1}{2} \left( \mathbf{z}_0 - \mathbf{z}_0^b \right)^\top 
        \mathbf{B}^{-1} \left( \mathbf{z}_0 - \mathbf{z}_0^b \right) \\
    &+ \frac{1}{2} \sum_{k=0}^{N} 
        \left( \mathbf{y}_k - Q_k \mathbf{z}_{0} \right)^\top 
        \mathbf{R}_k^{-1} \left( \mathbf{y}_k - Q_k \mathbf{z}_{0} \right) \\
    &- \frac{1}{2} \sum_{k=0}^{N} 
        \left( Q_k \mathbf{z}_{0} - Q_k \mathbf{z}_0^b \right)^\top 
        \mathbf{L}_k^{-1} \left( Q_k \mathbf{z}_{0} - Q_k \mathbf{z}_0^b \right). \label{eq:dc_4dvar_cost}
\end{split}
\end{equation}
subject to the model constraints of~\eqref{eq:strong_constraint}. Note that $\{Q_{k}\mathbf{z}_{0}^{b}\}_{k=0}^{N}$ defines the observations associated with the model-predicted trajectory of the prior mean. This trajectory plays a critical role in the regularization as opposed to standard 4D-Var, which only includes $\mathbf{z}_0^b$ in the first penalty term. Specifically, the inclusion of the predictability-aware regularization term~\eqref{eq:predict_norm} promotes data-consistency in that it serves as a type of ``targeted un-regularization'' that reduces the specific impact of any statistical bias in the prior only in the directions of the state space implicitly informed by the QoI, i.e., the observable portions of the trajectories. 

%\\\\\\\\\\\\\\\\\\\\\\\\\\\\\\\\\\\\\\\\\\\\\\\\\\\\\\\\\\\\\\\\\\\\\\\\\\\\\\\\\\\\\\\\\\\\\\\\\\\\\\\\\\\\\\\\\\\\\\\\\\\\\\\\\\\\\\\\\\\\\\\\\\\\\\\\\\\\\\\\\\\\\\\\\\\

\subsection{Hessian of the DC 4D-Var Cost function}
To compute the Hessian of~\eqref{eq:dc_4dvar_cost}, we consider the second derivative of each term individually. 
The first term, corresponding to the background mismatch, is quadratic in $ \mathbf{z}_0 $ and yields a constant second derivative:
\begin{equation}
    \nabla^2\left(\frac{1}{2}\left\|\mathbf{z}_0 - \mathbf{z}_0^b\right\|_{\mathbf{B}^{-1}}^2\right) = \mathbf{B}^{-1}.
\end{equation}
For the data misfit term, we define the residual $ \mathbf{r}_k(\mathbf{z}_0) := \mathbf{y}_k - Q_k(\mathbf{z}_0) $ and let $ \mathbf{J}_k = \frac{\partial Q_k(\mathbf{z}_0)}{\partial \mathbf{z}_0} $ denote the Jacobian of the QoI map. Then, 
\begin{equation}
    \nabla^2\left(\frac{1}{2} \mathbf{r}_k^\top \mathbf{R}_k^{-1} \mathbf{r}_k\right)
    = \mathbf{J}_k^\top \mathbf{R}_k^{-1} \mathbf{J}_k + (\nabla^2 \mathbf{r}_k)^\top \mathbf{R}_k^{-1} \mathbf{r}_k.
\end{equation}
Utilizing the Gauss-Newton approximation, we neglect the second-order term $ (\nabla^2 \mathbf{r}_k)^\top \mathbf{R}_k^{-1} \mathbf{r}_k $, giving:
\begin{equation}
    \nabla^2 \left(\frac{1}{2}\|\mathbf{r}_k\|_{\mathbf{R}_k^{-1}}^2\right) \approx \mathbf{J}_k^\top \mathbf{R}_k^{-1} \mathbf{J}_k.
\end{equation}
For the predictability-aware term, we define the correction residual as $ \mathbf{q}_k(\mathbf{z}_0) := Q_k(\mathbf{z}_0) - Q_k(\mathbf{z}_0^b) $.
Following the same steps as with the data misfit term, its second derivative is approximated as
\begin{equation}
    \nabla^2 \left(\frac{1}{2} \|\mathbf{q}_k\|_{\mathbf{L}_k^{-1}}^2 \right) \approx \mathbf{J}_k^\top \mathbf{L}_k^{-1} \mathbf{J}_k.
\end{equation}
Combining the above, the full Gauss-Newton approximation to the Hessian becomes:
\begin{equation}
    \nabla^2 \mathcal{J}_{\mathrm{DC}}(\mathbf{z}_0) \approx \mathbf{B}_k^{-1}
    + \sum_{k=0}^N \mathbf{J}_k^\top \left( \mathbf{R}_k^{-1} - \mathbf{L}_k^{-1} \right) \mathbf{J}_k.
\end{equation}

\paragraph{Special Case: Linear Model}

If the forward and observation operators $ \mathcal{M}_k $ and $ \mathcal{H}_k $ are linear, then the QoI map satisfies $ Q_k(\mathbf{z}_0) = \mathcal{H}_k \mathcal{M}_{k:0} \mathbf{z}_0 $, and $ \mathbf{J}_k = \mathcal{H}_k \mathcal{M}_{k:0} $ is constant. 
In that case, the exact Hessian is:
\begin{equation}
    \nabla^2 \mathcal{J}_{\mathrm{DC}}(\mathbf{z}_0) =
    \mathbf{B}_k^{-1} + \sum_{k=0}^N \mathcal{M}_{k:0}^\top \mathcal{H}_k^\top \left( \mathbf{R}_k^{-1} - \mathbf{L}_k^{-1} \right) \mathcal{H}_k \mathcal{M}_{k:0}.
\end{equation}

\paragraph{General Case: Nonlinear Model}

If $ \mathcal{M}_k $ or $ \mathcal{H}_k $ are nonlinear, then the Jacobian $ \mathbf{J}_k = \mathbf{H}_k \mathbf{M}_k $ varies with $ \mathbf{z}_0 $ and must be recomputed at each iteration. 
Nevertheless, the approximate Hessian retains the same algebraic structure:
\begin{equation}
    \nabla^2 \mathcal{J}_{\mathrm{DC}}(\mathbf{z}_0) \approx
    \mathbf{B}_k^{-1} + \sum_{k=0}^N \mathbf{M}_k^\top \mathbf{H}_k^\top \left( \mathbf{R}_k^{-1} - \mathbf{L}_k^{-1} \right) \mathbf{H}_k \mathbf{M}_k  \label{eq:dc_hessian}
\end{equation}
where the tangent linear model $ \mathbf{M}_k = \frac{\partial \mathcal{M}_{k:0}}{\partial \mathbf{z}_0} $ is evaluated at $ \mathbf{z}_0 $.
%\\\\\\\\\\\\\\\\\\\\\\\\\\\\\\\\\\\\\\\\\\\\\\\\\\\\\\\\\\\\\\\\\\\\\\\\\\\\\\\\\\\\\\\\\\\\\\\\\\\\\\\\\\\\\\\\\\\\\\\\\\\\\\\\\\\\\\\\\\\\\\\\\\\\\\\\\\\\\\\\\\\\\\\\\\\
\subsection{Updated Covariance Structure}
In a manner analogous to traditional 4D-Var, the posterior covariance in the DC 4D-Var formulation is given by the inverse of the Hessian of $ \mathcal{J}_{\mathrm{DC}} $. 
Under the Gauss-Newton approximation, the updated posterior covariance takes the form:
\begin{equation}
    \mathbf{P}_{\mathrm{DC}}^{\text{up}} =
    \left[
    \mathbf{B}^{-1}
    + \sum_{k=0}^N \mathbf{M}_k^\top \mathbf{H}_k^\top \left( \mathbf{R}_k^{-1} - \mathbf{L}^{-1} \right) \mathbf{H}_k \mathbf{M}_k
    \right]^{-1}.
\end{equation}
We can express this more compactly as:
\begin{equation}
    \mathbf{P}_{\mathrm{DC}}^{\text{up}} =
    \left[ \mathbf{B}^{-1} + \sum_{k=0}^N \mathbf{M}_k^\top \mathbf{H}_k^\top \mathbf{W}_k^{-1} \mathbf{H}_k \mathbf{M}_k \right]^{-1}.
\end{equation}
The ratio precision matrix $ \mathbf{W}_k^{-1} $ plays a critical role in adjusting the posterior uncertainty based on the relative informativeness of the observations and the model prediction. 
For instance, if $ \mathbf{R}_k = \mathbf{H}_k \mathbf{B} \mathbf{H}_k^\top $, then $ \mathbf{W}_k^{-1} = 0 $, and the observation contributes no additional information beyond what is already predicted by the model. 
If $ \mathbf{R}_k < \mathbf{H}_k \mathbf{B} \mathbf{H}_k^\top $, then $ \mathbf{W}_k^{-1} \succ 0 $, indicating that the observation is more informative than the model and should be used to reduce uncertainty. Conversely, if $ \mathbf{R}_k > \mathbf{H}_k \mathbf{B} \mathbf{H}_k^\top $, then $ \mathbf{W}_k^{-1} \prec 0 $, meaning the observation is less informative than the model’s prediction and should be down-weighted in the analysis. 
This, of course, results in a violation of the predictability assumption. 
Subsequently, the predictability-aware term provides a structure that ensures that observations are trusted only when they improve upon the model’s predicted uncertainty {\em if the predictability assumption holds.}
%\\\\\\\\\\\\\\\\\\\\\\\\\\\\\\\\\\\\\\\\\\\\\\\\\\\\\\\\\\\\\\\\\\\\\\\\\\\\\\\\\\\\\\\\\\\\\\\\\\\\\\\\\\\\\\\\\\\\\\\\\\\\\\\\\\\\\\\\\\\\\\\\\\\\\\\\\\\\\\\\\\\\\\\\\\\
\subsection{DC Weighted Mean Error 4D-Var (DC-WME 4D-Var)}
As previously mentioned, \cite{Pilosov2023} demonstrated that if $Q$ is defined by a weighted mean error (WME) map for linear observations, then the predictability assumption is guaranteed to hold after a sufficient amount of observable data are collected.
The subsequent MUD point estimate was shown to have several desirable properties.
For instance, uncertainties in the MUD estimate, quantified by covariances, were shown to decrease in the input directions informed by the measurement operators at a rate proportional to the number of observed data used for each measurement.
Additionally, propagating the MUD estimate through the solution operator to the model was shown to produce an unbiased estimate of the sample mean of the observed data\footnote{We refer the interested reader to equations (9)-(13) and Theorem 4.1 and Corollary 1 of~\cite{Pilosov2023} for more details.}.
This motivates the adoption of a similar Weighted Mean Error (WME) QoI map, denoted as $Q_{\mathrm{wme}}$ in this work, and defined as:
\begin{equation}
    Q_{\mathrm{wme}}\left(\mathbf{z}_0\right) := \frac{1}{\sqrt{N}} \sum_{k=1}^{N} \mathbf{R}_k^{-1/2}\left[\mathcal{H}_k \circ \mathcal{M}_{k:0}\left(\mathbf{z}_0\right)-\mathbf{y}_{k}\right].   \label{eq:wme_map}
\end{equation}
This definition captures the scaled residuals between predicted and observed data at time $t_{k,}$. Using this map we obtain the DC-WME 4D-Var cost function as
\begin{equation}
\begin{split}
    \mathcal{J}_{\mathrm{DC}}\left(\mathbf{z}_0\right) 
    &= \frac{1}{2}\left\|\mathbf{z}_0-\mathbf{z}_0^{\mathrm{b}}\right\|_{\mathbf{B}_0^{-1}}^2 
     + \frac{1}{2}\left\|Q_{\mathrm{wme}}\left(\mathbf{z}_0\right)\right\|^2 \\
    &\quad - \frac{1}{2}\left\|Q_{\mathrm{wme}}\left(\mathbf{z}_0\right) 
        - Q_{\mathrm{wme}}\left(\mathbf{z}_0^{\mathrm{b}}\right)\right\|_{\mathbf{L}_\mathrm{wme}^{-1}}^2 .     \label{eq:wme_4dvar_func}
\end{split}
\end{equation}
Note the second term utilizes the Euclidean norm (due to the construction of the WME map) and $\mathbf{L}_\mathrm{wme}$ denotes the predicted covariance associated with the WME map.

Before we continue, we first briefly compare and contrast the three separate cost functions presented thus far. 
Clearly, the first terms associated with the background knowledge in equations~\eqref{eq:4dvar_cost},~\eqref{eq:dc_4dvar_cost}, and~\eqref{eq:wme_4dvar_func} are all identical. 
By expanding the Euclidean norm, we see that the data-misfit term (i.e., the second term) in~\eqref{eq:wme_4dvar_func} is just a scaled version (by a factor of $1/N$) of the data-misfit terms in~\eqref{eq:4dvar_cost} and~\eqref{eq:dc_4dvar_cost}.
However, the predictability-aware regularization terms are fundamentally different in~\eqref{eq:dc_4dvar_cost} and~\eqref{eq:wme_4dvar_func}.
Specifically, the incorporation of the weighted residuals in the WME map implies that the observed data play a specific role in the predictability-aware regularization.
Just as the WME map played a significant role in reducing uncertainties in MUD point estimates in both \cite{Pilosov2023} and \cite{del-Castillo-Negrete2024} over a naive implementation of MUD point estimation, we similarly demonstrate it improves the performance of DC-based 4D-Var. 

To make the following analysis more tractable, we first apply the standard assumption that the state vectors are observed with independent observational errors following the observation model:
\begin{equation}
    \mathbf{y}_k=Q_{k}\left(\mathbf{z}_{0}^\dagger\right)+\mathbf{\varepsilon}_{k}, \quad \mathbf{\varepsilon}_{k} \sim \mathcal{N}\left(\boldsymbol{0}, \sigma_{\mathrm{obs},k}^2 \boldsymbol{I}_{d\times d}\right)    \label{eq:observation_equation}
\end{equation}
where we again make use of the shorthand notation $Q_k=\mathcal{H}_k\circ\mathcal{M}_{k:0}$.
Note that since the observed data follow the true system state obtained by propagating some true (but unknown) $ \mathbf{z}_{0}^{\dagger} $ as assumed in~\eqref{eq:observation_equation}, this subsequently implies that the observable QoI data (i.e., the data obtained by observing the QoI for the true trajectory) follows a $ \mathcal{N}\left(\mathbf{0}_{d \times 1}, \mathbf{I}_{d \times d}\right) $ distribution. 
This approach subsequently keeps the observed QoI distribution unchanged regardless of the number of measurements, simplifying algorithmic implementation. 
Another common assumption is that the observational covariance is temporally invariant, in which case we can drop the $k$ index on $\sigma_{\text{obs},k}$, and from~\eqref{eq:wme_map} it follows that
\begin{equation}
    \mathbf{L}_\mathrm{wme} \approx \frac{N}{\sigma_{\text{obs}}^2}  \mathbf{Q}_1 \mathbf{B} \mathbf{Q}_1^\top \in \mathbb{R}^{d \times d}  \label{eq:wme_pred_cov}
\end{equation}
where $\mathbf{Q}_1$ denotes the linear approximation to the map $Q_1$. 
%\\\\\\\\\\\\\\\\\\\\\\\\\\\\\\\\\\\\\\\\\\\\\\\\\\\\\\\\\\\\\\\\\\\\\\\\\\\\\\\\\\\\\\\\\\\\\\\\\\\\\\\\\\\\\\\\\\\\\\\\\\\\\\\\\\\\\\\\\\\\\\\\\\\\\\\\\\\\\\\\\\\\\\\\\\\
\subsection{Gradient of the DC-WME 4D-Var Cost function}
For notational brevity, we define $Q := Q_{\mathrm{wme}}(\mathbf{z}_0)$ and $ Q^{\mathrm{b}} := Q_{\mathrm{wme}}(\mathbf{z}_0^{\mathrm{b}}) $. 
Assuming sufficient regularity, the gradient of the cost function with respect to the initial condition $ \mathbf{z}_0 $ is given by
\begin{equation}
\begin{split}
\nabla \mathcal{J}_{\mathrm{DC}}(\mathbf{z}_0) 
    &= \mathbf{B}_0^{-1} \left( \mathbf{z}_0 - \mathbf{z}_0^{\mathrm{b}} \right) \\
    &\quad + \left( D_{\mathbf{z}_0} Q(\mathbf{z}_0)\right)^\top Q \\
    &\quad - \left( D_{\mathbf{z}_0} Q(\mathbf{z}_0) \right)^\top 
        \mathbf{L}_{\mathrm{wme}}^{-1} \left( Q - Q^{\mathrm{b}} \right).
\end{split}
\end{equation}
where $ D_{\mathbf{z}_0} Q(\mathbf{z}_0) $ denotes the Jacobian of the WME map with respect to the initial condition. 
If the observation operator $ \mathcal{H} $ and the forward model $ \mathcal{M}_{0} $ are differentiable, and $ \mathcal{H} $ is linear, then the Jacobian takes the form
\begin{equation}
D_{\mathbf{z}_0} Q\left(\mathbf{z}_0\right)=\frac{1}{\sqrt{N}} \sum_{k=1}^{N} \frac{1}{\sigma_{\mathrm{obs}, k}} \mathcal{H}_k \circ \mathbf{M}_{k: 0}\left(\mathbf{z}_0\right)     \label{eq:wme_qoi_jacobian}
\end{equation}
where $\mathbf{M}_{k: 0}\left(\mathbf{z}_0\right):=D_{\mathbf{z}_0} \mathcal{M}_{k: 0}\left(\mathbf{z}_0\right)$ represents the sensitivity of the model state at time $t_k$ with respect to the initial condition.

%\\\\\\\\\\\\\\\\\\\\\\\\\\\\\\\\\\\\\\\\\\\\\\\\\\\\\\\\\\\\\\\\\\\\\\\\\\\\\\\\\\\\\\\\\\\\\\\\\\\\\\\\\\\\\\\\\\\\\\\\\\\\\\\\\\\\\\\\\\\\\\\\\\\\\\\\\\\\\\\\\\\\\\\\\\\
\subsection{Data Consistent 4D-Var Adjoint}
%\\\\\\\\\\\\\\\\\\\\\\\\\\\\\\\\\\\\\\\\\\\\\\\\\\\\\\\\\\\\\\\\\\\\\\\\\\\\\\\\\\\\\\\\\\\\\\\\\\\\\\\\\\\\\\\\\\\\\\\\\\\\\\\\\\\\\\\\\\\\\\\\\\\\\\\\\\\\\\\\\\\\\\\\\\\
To derive the adjoint equations associated with the DC 4D-Var cost function of~\eqref{eq:dc_4dvar_cost}, we introduce Lagrange multipliers $ \boldsymbol{\lambda}_{k+1} $ to enforce the discrete model constraints $ \mathbf{z}_{k+1} = \mathcal{M}_k(\mathbf{z}_k) $ for each time index $ k = 0, \dots, N-1 $ as well as $\boldsymbol{\lambda}_0$ to enforce the initial state. 
We begin with the augmented Lagrangian function:
\begin{equation}
\begin{split}
    \mathcal{L} 
    &= \mathcal{J}_{\mathrm{DC}}(\mathbf{z}_{0:N}) \\
    &\quad + \sum_{k=0}^{N-1} \boldsymbol{\lambda}_{k+1}^\top 
        \left( \mathbf{z}_{k+1} - \mathcal{M}_k(\mathbf{z}_k) \right).  \label{eq:q_lagrange}
\end{split}
\end{equation}
The adjoint equations follow from setting the first variation of the Lagrangian with respect to each state variable, $ \mathbf{z}_k $, equal to zero. 
We utilize the shorthand notation $Q_k := Q_k(\mathbf{z}_k)$ and $Q_k^{\mathrm{b}} := Q_k(\mathbf{z}_0^{\mathrm{b}})$ and define the observation and predictability residuals as $\mathbf{r}_k := \mathbf{y}_k - Q_k$ and $\mathbf{q}_k := Q_k - Q_k^{\mathrm{b}}$, respectively. 
We distinguish three cases in the derivation. For the terminal time $ k = N $, the state $ \mathbf{z}_N $ appears only in the data misfit and predictability terms, as well as in the constraint for $ \mathbf{z}_N $. Differentiating the Lagrangian with respect to $\mathbf{z}_N$ and setting it to zero yields 
\begin{equation}
    \boldsymbol{\lambda}_N = 
    (\nabla Q_N)^\top \left[
    \mathbf{R}_N^{-1} \mathbf{r}_N
    - \mathbf{L}_N^{-1} \mathbf{q}_N
    \right].
\end{equation}
For intermediate times $ 0 < k < N $, the variable $ \mathbf{z}_k $ contributes to the data misfit and predictability terms at time $ k $, and also appears in the model constraints for both $ \mathbf{z}_k $ and $ \mathbf{z}_{k+1} $. 
Differentiating the Lagrangian with respect to $ \mathbf{z}_k $ and setting it zero yields:
\begin{equation}
    \boldsymbol{\lambda}_k = (\nabla Q_k)^\top 
    \left[
    \mathbf{R}_k^{-1} \mathbf{r}_{k}
    - \mathbf{L}_k^{-1} \mathbf{q}_{k}
    \right]
    + (\nabla \mathcal{M}_k)^\top \boldsymbol{\lambda}_{k+1}.
\end{equation}
The initial state $ \mathbf{z}_0 $ appears in the background prior, the observation and predictability terms at time zero, and the constraint governing the next state. 
Differentiating the Lagrangian with respect to $\mathbf{z}_0$ and setting it to zero yields:

\begin{equation}
    \boldsymbol{\lambda}_0 = 
-\mathbf{B}^{-1} (\mathbf{z}_0 - \mathbf{z}_0^{\mathrm{b}}).
\end{equation}

%\\\\\\\\\\\\\\\\\\\\\\\\\\\\\\\\\\\\\\\\\\\\\\\\\\\\\\\\\\\\\\\\\\\\\\\\\\\\\\\\\\\\\\\\\\\\\\\\\\\\\\\\\\\\\\\\\\\\\\\\\\\\\\\\\\\\\\\\\\\\\\\\\\\\\\\\\\\\\\\\\\\\\\\\\\\
\subsection{Weighted Mean Error QoI Map}
In the case where the cost function utilizes the WME map at each time $ t_k $, we note that this map aggregates scaled misfits across all observations into a single quantity of interest. To derive the gradient of $ \mathcal{J}_{\mathrm{DC}} $, we once again introduce Lagrange multipliers and the gradient is computed using a forward–backward pass. In the forward pass, we initialize $ \mathbf{z}_0 $ and propagate the system forward using the model dynamics where the corresponding WME map $ Q_{\mathrm{wme}}(\mathbf{z}_0) $ is evaluated for each $k=0,\ldots, N-1$. 
The backward pass then computes the adjoint variables $ \boldsymbol{\lambda}_k $ using a discrete adjoint recursion. The recursion is initialized with $ \boldsymbol{\lambda}_{N+1} := 0 $.  For each time step $ k = N, N-1, \dots, 0 $, we compute the Jacobian of the WME map given by~\eqref{eq:wme_qoi_jacobian}:
\begin{equation}
    \mathbf{J}_k := D_{\mathbf{z}_k} Q_{\mathrm{wme}}(\mathbf{z}_0) = \frac{1}{\sqrt{N}} \mathbf{R}^{-\frac{1}{2}}_{k}\mathcal{H}_{k}
\end{equation} which is constant in $ \mathbf{z}_0 $ since the WME map is affine in the state variable. 
The local contribution to the gradient is given by
\begin{equation}
\nabla_{\mathbf{z}_k} \mathcal{J}_k =
      \mathbf{J}_k^\top \left[
         Q_{\mathrm{wme}}\left(\mathbf{z}_0\right) - \mathbf{L}_{\mathrm{wme}}^{-1} \left(Q_{\mathrm{wme}}\left(\mathbf{z}_0 -\mathbf{z}_0^{\mathrm{b}}\right)\right)
        \right]
\end{equation}
which captures the difference between the unweighted WME term and the predictability penalty. 
The adjoint variables are then updated recursively as
\begin{equation}
    \boldsymbol{\lambda}_k = 
    D_{\mathbf{z}_k} \mathcal{M}_k^\top \boldsymbol{\lambda}_{k+1}
    + \nabla_{\mathbf{z}_k} \mathcal{J}_k,
\end{equation}
proceeding backward in time from $ k = N $ to $ k = 0 $. 
Finally, the last adjoint variable is computed as 
\begin{equation}
    \boldsymbol{\lambda}_0 = -\mathbf{B}_0^{-1} (\mathbf{z}_0 - \mathbf{z}_0^{\mathrm{b}}).
\end{equation}
%\\\\\\\\\\\\\\\\\\\\\\\\\\\\\\\\\\\\\\\\\\\\\\\\\\\\\\\\\\\\\\\\\\\\\\\\\\\\\\\\\\\\\\\\\\\\\\\\\\\\\\\\\\\\\\\\\\\\\\\\\\\\\\\\\\\\\\\\\\\\\\\\\\\\\\\\\\\\\\\\\\\\\\\\\\\
\section{Computational Considerations}\label{sec:comp_considerations}
%\\\\\\\\\\\\\\\\\\\\\\\\\\\\\\\\\\\\\\\\\\\\\\\\\\\\\\\\\\\\\\\\\\\\\\\\\\\\\\\\\\\\\\\\\\\\\\\\\\\\\\\\\\\\\\\\\\\\\\\\\\\\\\\\\\\\\\\\\\\\\\\\\\\\\\\\\\\\\\\\\\\\\\\\\\\
\subsection{The Predictability Assumption}\label{subsec:pred_cov_inflate}

Consider again~\eqref{eq:wme_pred_cov}, and let the background (initial) covariance matrix be $ \mathbf{B} = \sigma_b^2 \mathbf{I} \in \mathbb{R}^{n \times n} $, which encodes an isotropic structure with variance $ \sigma_b^2 $. Similarly, the observation covariance matrix is defined as $ \mathbf{R}_k = \sigma_{\text{obs}}^2 \mathbf{I} \in \mathbb{R}^{d \times d} $, also assuming isotropy with variance $ \sigma_{\text{obs}}^2 $. 
Finally, the QoI at time $ k $ is given by the composition QoI map $ \mathbf{Q}_k = \mathcal{H}_k \circ\mathcal{M}_{k:0} \in \mathbb{R}^{d \times n} $, where $ \mathcal{H}_k $ is the linear or nonlinear observation matrix and $ \mathcal{M}_k $ is the model operator. Then, the minimum eigenvalue of $ \mathbf{L}_k $ is given by
\begin{equation}
    \lambda_{\min}(\mathbf{L}_k) = \frac{N   \sigma_b^2}{\sigma_{\text{obs}}^2}   \lambda_{\min}(\mathbf{Q}_k \mathbf{Q}_k^\top).
\end{equation}
We again recall that \cite{Pilosov2023} shows that for sufficiently large $N$:
\begin{equation}
    \lambda_{\min}(\mathbf{L}_k) \geq  \lambda_{\max}(\mathbf{R}_k) =  \sigma_{\text{obs}}^2.
\end{equation}
While this guarantees that the predictability assumption holds, it is often far too conservative in practice since it is generally unnecessary to ensure that the slowest direction of asymptotic decay of the observed distribution is faster than the asymptotic decay of the predicted distribution in any direction.
Subsequently, we introduce the positive scaling factor $\gamma$, typically chosen much less than $1$, so that
\begin{equation}
    \lambda_{\min}(\mathbf{L}_k) \geq   \gamma \cdot\lambda_{\max}(\mathbf{R}_k) =   \gamma \cdot \sigma_{\text{obs}}^2.
\end{equation}
Then, substituting this expression for $ \lambda_{\min}(\mathbf{L}_k) $, and multiplying both sides by $ \sigma_{\text{obs}}^2 $, we get:
\begin{equation}
N   \sigma_b^2   \lambda_{\min}(\mathbf{Q}_k \mathbf{Q}_k^\top) \geq \gamma \cdot   \sigma_{\text{obs}}^4
\end{equation}
Solving for $ \sigma_b^2 $, we obtain:
\begin{equation}
\sigma_b^2 \geq \frac{\gamma  \cdot \sigma_{\text{obs}}^4}{N   \lambda_{\min}(\mathbf{Q}_k \mathbf{Q}_k^\top)}. \label{eq:eig_low_bound}
\end{equation}
This ensures the background variance is sufficiently large to avoid overconfidence in data directions that are weakly informed by the dynamics while still maintaining tighter variances in the prediction.
%\\\\\\\\\\\\\\\\\\\\\\\\\\\\\\\\\\\\\\\\\\\\\\\\\\\\\\\\\\\\\\\\\\\\\\\\\\\\\\\\\\\\\\\\\\\\\\\\\\\\\\\\\\\\\\\\\\\\\\\\\\\\\\\\\\\\\\\\\\\\\\\\\\\\\\\\\\\\\\\\\\\\\\\\\\\
\subsection{Computational Cost}\label{subsec:comp_cost}
In practice, the dominant computational cost is typically model propagation, especially in high-dimensional systems such as those arising from PDE-based models. 
Consequently, evaluating any of the cost functions is often comparable to the cost of a single forward model run, supplemented by additional linear algebra operations to compute the misfit terms. 

The DC 4D-Var formulation introduces a modest increase in the number of observation operator evaluations compared to the traditional 4D-Var approach due to the inclusion of the additional predictability term in the cost function. 
In the DC-WME 4D-Var case, the dominant computational cost arises from the repeated model evaluations required by the WME QoI map as shown below.

Let $n$ denote the dimension of the state space, and let $m_k$ represent the dimension of the observation at time $t_k$. The model is propagated forward over $N$ time steps (from $t_0$ to $t_N$). 
Let $\mathcal{I} = \{k_1, \ldots, k_K\} \subseteq \{0, \ldots, N\}$ denote the set of observation time indices, with $K = |\mathcal{I}|$. 
The cost of one forward model step $\mathcal{M}_k$ is denoted $C_\mathcal{M}$, and the cost of evaluating the observation operator $\mathcal{H}_k$ is denoted $C_\mathcal{H}(n, m_k)$. The total cost of computing the forward model trajectory is 
\[
\mathcal{O}(N \cdot C_\mathcal{M}),
\]
while the total cost of computing the observations is
\[
\mathcal{O}\left( \sum_{k \in \mathcal{I}} C_\mathcal{H}(n, m_k) \right).
\]
Additional quadratic terms contribute to the total computational cost. 
The background term incurs a cost of $\mathcal{O}(n^2)$ if $\mathbf{B}^{-1}$ is dense. 
For the observation misfit terms, assuming that each $\mathbf{R}_k$ is diagonal or sparse, the cost per time step is $\mathcal{O}(m_k^2)$, leading to an overall cost of
\[
    \mathcal{O}\left( \sum_{k \in \mathcal{I}} m_k^2 \right).
\]
For the DC 4D-Var cost function, the predictability portion of the cost function creates an additional cost of
\[
    \mathcal{O}\left( \sum_{k \in \mathcal{I}} 2 \cdot C_{\mathcal{H}}(n, m_k) \right).
\]
Under the assumption that the covariance structure of $\mathbf{L}_k$ is similar to that of $\mathbf{R}_k$, the predictability misfit term adds an additional cost of
\[
    \mathcal{O}\left( \sum_{k \in \mathcal{I}} m_k^2 \right).
\]

Evaluating $Q_{\mathrm{wme}}(\mathbf{z}_0)$ involves two main computational tasks. 
First, it requires $N$ model propagations from the initial time $t_0$ to each time $t_k \in \mathcal{I}$. 
Second, it necessitates applying the observation operator at each of these times. 
As a result, the total cost of evaluating all $Q_{\mathrm{wme}}(\mathbf{z}_0)$ terms is
\[
    \mathcal{O}\left( K \cdot N \cdot C_{\mathcal{M}} \right) + \mathcal{O}\left( K \cdot N \cdot C_{\mathcal{H}}(n, m_k) \right).
\]

In addition to these map evaluations, the cost function includes norm computations. 
For each term of the form $\left\| Q_{\mathrm{wme}}(\mathbf{z}_0) \right\|^2$, the cost is $\mathcal{O}(m_k)$. 
For each predictability term involving the difference $Q_{\mathrm{wme}}(\mathbf{z}_0) - Q_{\mathrm{wme}}(\mathbf{z}_0^{\mathrm{b}})$, and assuming application of the inverse covariance $\mathbf{L}_k^{-1}$, the cost becomes $\mathcal{O}(m_k^2)$. 
Therefore, the total cost associated with all norm computations is
\[
    \mathcal{O}\left( \sum_{k \in \mathcal{I}} m_k^2 \right).
\]

To summarize, the standard 4D-Var method has a total computational cost of
\begin{equation*}
    \mathcal{O}(N \cdot C_{\mathcal{M}}) + \mathcal{O}\left( \sum\limits_{k \in \mathcal{I}} C_{\mathcal{H}}(n, m_k) \right) + \mathcal{O}(n^2)+ \mathcal{O}\left( \sum\limits_{k \in \mathcal{I}} m_k^2 \right),
\end{equation*}
the DC 4D-Var method has a total computational cost of
\begin{equation*}
    \mathcal{O}(N \cdot C_{\mathcal{M}}) + \mathcal{O}\left( \sum\limits_{k \in \mathcal{I}} 2 \cdot C_{\mathcal{H}}(n, m_k) \right)   + \mathcal{O}(n^2) + \mathcal{O}\left( \sum\limits_{k \in \mathcal{I}} m_k^2 \right),
\end{equation*}
and the DC-WME 4D-Var method has a total computational cost of
\begin{equation*}
\begin{split}
    &\mathcal{O}(K \cdot N \cdot C_{\mathcal{M}}) 
     + \mathcal{O}(K \cdot N \cdot C_{\mathcal{H}}(n, m_k)) \\
    &\quad + \mathcal{O}(n^2) 
     + \mathcal{O}\!\left( \sum\limits_{k \in \mathcal{I}} m_k^2 \right).
\end{split}
\end{equation*}
To summarize, the DC-WME 4D-Var formulation is more expensive than traditional 4D-Var since each evaluation of the cost function involves a nested summation over all observation times. However, this can be mitigated to some degree by employing certain strategies such as reusing stored model trajectories or implementing vectorized evaluations of the composite map $ \mathcal{M}_{k:0}(\mathbf{z}_0) $.
%\\\\\\\\\\\\\\\\\\\\\\\\\\\\\\\\\\\\\\\\\\\\\\\\\\\\\\\\\\\\\\\\\\\\\\\\\\\\\\\\\\\\\\\\\\\\\\\\\\\\\\\\\\\\\\\\\\\\\\\\\\\\\\\\\\\\\\\\\\\\\\\\\\\\\\\\\\\\\\\\\\\\\\\\\\\
\section{Design of Numerical Experiments}
%\\\\\\\\\\\\\\\\\\\\\\\\\\\\\\\\\\\\\\\\\\\\\\\\\\\\\\\\\\\\\\\\\\\\\\\\\\\\\\\\\\\\\\\\\\\\\\\\\\\\\\\\\\\\\\\\\\\\\\\\\\\\\\\\\\\\\\\\\\\\\\\\\\\\\\\\\\\\\\\\\\\\\\\\\\

\subsection{Computational Implementation}
Data assimilation problems in meteorology and oceanography typically involve high-dimensional state spaces, often exceeding $10^7$ degrees of freedom. Thus, the practical implementation of any 4D-Var scheme relies critically on the rapid convergence of memory-efficient gradient-based algorithms for large-scale unconstrained minimization. 
In this setting, conjugate gradient and limited-memory quasi-Newton (LMQN) methods prove particularly effective, as they require storage of only a small number of vectors from recent iterations. 
Among LMQN methods, the Limited-memory Broyden–Fletcher–Goldfarb–Shanno (L-BFGS) algorithm consistently demonstrates strong performance and has become one of the most widely used approaches in operational 4D-Var implementations~\citep{Liu1989,Nash1991,Zou1993}.

For the two ODE examples shown in Section~\ref{sec:numerical_results}, the code was written using JAX~\citep{jax2018github} to leverage its built-in automatic differentiation capabilities within the solvers that allows us to utilize the Quasi-Newton L-BFGS method to perform all optimization routines across all three 4D-Var algorithms. 

Storm surge models typically rely on the Shallow Water Equations (SWE). 
For the SWE example shown in Section~\ref{sec:numerical_results}, we utilize SWEMniCSx~\citep{Dawson2024}, a new SWE software package, which is a Python-based finite element solver that utilizes the FEniCSx~\citep{Scroggs2022ty,Baratta2023,Alnaes2014,Scroggs2022} library for the 2D depth-integrated SWE. 
It provides a range of finite element methods. B
By using FEniCSx, novel numerical schemes are easily employed, including a fully implicit DG solver in time and a mixed DG-CG scheme. 
This flexibility makes SWEMniCSx an ideal forward modeling software for adjoint-based variational data assimilation algorithms.

It is worth noting that in operational practice, meteorological and oceanographic 4D-Var systems implement iterative solvers like L-BFGS or conjugate gradient within an inner loop, and often perform multiple outer loops to update the linearization, e.g., see~\citep{Navon1987,Li1993}. 
Although rigorous global convergence guarantees remain out of reach in the nonlinear case, using a good initial guess, typically a short range forecast, combined with the physical constraints imposed by the model, is often sufficient to steer the algorithm toward a physically meaningful local minimum, e.g., see~\citep{Mahfouf2000, Rabier2000, Janiskova1999}. 

%\\\\\\\\\\\\\\\\\\\\\\\\\\\\\\\\\\\\\\\\\\\\\\\\\\\\\\\\\\\\\\\\\\\\\\\\\\\\\\\\\\\\\\\\\\\\\\\\\\\\\\\\\\\\\\\\\\\\\\\\\\\\\\\\\\\\\\\\\\\\\\\\\\\\\\\\\\\\\\\\\\\\\\\\\\\
\subsection{Constructing an Observation Model}  

As is common in data assimilation studies, we conduct identical twin experiments by generating the sequence of true hidden states $ \mathbf{z}_{0}^{\mathrm{true}} $ and synthetic observations $ \mathbf{y}^{\mathrm{obs}}_k \sim \mathcal{N}\left(\mathcal{H}_k\mathbf{z}_{0}^{\mathrm{true}}, \sigma_{\mathrm{obs}} \mathbf{I}_{d}\right) $ using the same forecast and observation models.

%\\\\\\\\\\\\\\\\\\\\\\\\\\\\\\\\\\\\\\\\\\\\\\\\\\\\\\\\\\\\\\\\\\\\\\\\\\\\\\\\\\\\\\\\\\\\\\\\\\\\\\\\\\\\\\\\\\\\\\\\\\\\\\\\\\\\\\\\\\\\\\\\\\\\\\\\\\\\\\\\\\\\\\\\\\\
\subsection{Background Covariance Inflation and Predictability}\label{subsubsec:inflate_bound} 

Data assimilation algorithms often apply covariance inflation  (~\cite{Anderson1999, Anderson2007, Anderson2009, Bocquet2012, Miyoshi2011, Luo2013}) to counteract the excessive variance reduction caused by spurious correlations during the update. 
Motivated by this, we describe the background covariance matrix using the relationship
\begin{equation}
    \mathbf{B} = \alpha * \mathbf{I}    \label{eq:background_cov_inflate}
\end{equation}
where $\alpha$ is an inflation tuning parameter that is determined heuristically. 
Recalling Section~\ref{sec:Methodology}~\ref{subsec:pred_cov_inflate}, $\alpha$ may be interpreted as the covariance inflation parameter required for the predictability assumption to be satisfied. 
%\\\\\\\\\\\\\\\\\\\\\\\\\\\\\\\\\\\\\\\\\\\\\\\\\\\\\\\\\\\\\\\\\\\\\\\\\\\\\\\\\\\\\\\\\\\\\\\\\\\\\\\\\\\\\\\\\\\\\\\\\\\\\\\\\\\\\\\\\\\\\\\\\\\\\\\\\\\\\\\\\\\\\\\\\\\
Specifically, from~\eqref{eq:eig_low_bound}, it is clear that the parameter $ \gamma $ plays a critical role in controlling the minimum required background variance $ \sigma_b^2 $. 
Conceptually, $ \gamma $ serves as a safety factor that prevents the assimilation process from becoming overconfident in directions where the model dynamics provide limited information. 
In situations where the predictability precision operator $ \mathbf{L}_k $ is unavailable or undefined, $ \gamma $ cannot be computed directly from theory. 
Instead, it is chosen based on different heuristics. 
Choosing a small value such as $\gamma = 0.01$ or $\gamma = 0.1$ will generally ensure that the background covariance remains sufficiently large, even in directions where the observation operator is highly informative while mitigating overfitting and numerical instability in poorly constrained directions. 
Alternatively, numerical tuning treats $\gamma$ as a tunable hyperparameter: multiple candidate values are tested, and the corresponding assimilation performance (e.g., RMSE) is evaluated to identify a value that balances model trust with observational influence. 
Yet another heuristic is based on inverse calibration where a target or assumed value of $\sigma_b^2$ is known so that one can estimate $\gamma$ by rearranging the inequality~\eqref{eq:eig_low_bound} and solving for it using the observed eigenvalue of the Gram matrix associated with the quantity of interest (QoI) map. This provides a data-driven method for calibrating the lower bound. 
Ultimately, the choice of $ \gamma $ reflects a modeling judgment about the relative confidence in observational data versus the information content and uncertainty structure of the model dynamics and observations.
%\\\\\\\\\\\\\\\\\\\\\\\\\\\\\\\\\\\\\\\\\\\\\\\\\\\\\\\\\\\\\\\\\\\\\\\\\\\\\\\\\\\\\\\\\\\\\\\\\\\\\\\\\\\\\\\\\\\\\\\\\\\\\\\\\\\\\\\\\\\\\\\\\\\\\\\\\\\\\\\\\\\\\\\\\\\
\section{Numerical Results}\label{sec:numerical_results}
Here, we present three examples. The first involves a state-of-the-art simulation model for the shallow water equations (SWE) on a domain modeling a sloped beach. The purpose of that example is two-fold. First, the complexity of the computational model serves as a sufficiently complex test of the correct implementation of the adjoint-based methods. The tidal forcing and structure of the domain leads to periodic and smooth behavior for which any reasonable method is expected to give similar results. This leads to the second purpose, which is a ``sanity-check'' that all three methods are correctly implemented and produce comparable results for a variety of observational networks.
The second and third examples utilize benchmark systems (Lorenz-63 and Lorenz-96) to highlight differences in the performance of the methods when the dynamical systems exhibit more chaotic behavior. 
\begin{figure*}[t]
\centerline{\includegraphics[width=\textwidth]{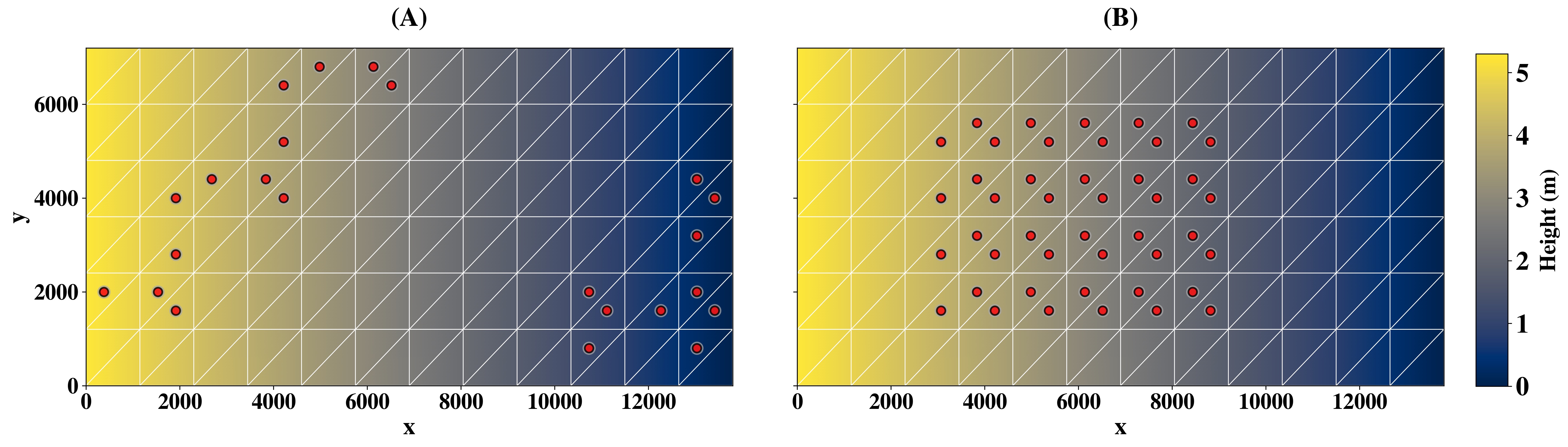}}
\caption{Spatial configurations (A) and (B) of observation stations (shown as red dots) used in the SWE experiments. Stations are distributed along transects perpendicular to the shoreline within a two-dimensional domain discretized into 144 triangular elements. These configurations are used as representative examples to demonstrate that all three assimilation methods perform equivalently in a well-conditioned, periodic setting, thereby serving as a baseline verification of correct implementation.}
\label{fig:station_locs}
\end{figure*}
%\\\\\\\\\\\\\\\\\\\\\\\\\\\\\\\\\\\\\\\\\\\\\\\\\\\\\\\\\\\\\\\\\\\\\\\\\\\\\\\\\\\\\\\\\\\\\\\\\\\\\\\\\\\\\\\\\\\\\\\\\\\\\\\\\\\\\\\\\\\\\\\\\\\\\\\\\\\\\\\\\\\\\\\\\\\
\subsection{Shallow Water Equations (SWE)}\label{sec:swe}
The two-dimensional depth-integrated SWE, given by
 \begin{equation}
        \partial_t \mathbf{Q} + div(\mathbf{F(\mathbf{U})}) + \mathbf{f(\mathbf{U})} = 0,   \label{eqn:conservative_swe2}
\end{equation}
are widely used in meteorology and oceanography to test new algorithms, as they capture most of the physical DoF found in more advanced three-dimensional primitive equation models. 
Here, $\mathbf{Q} = (H, uH, vH)^\top$ is the vector of conserved variables and express the primitive variables as the vector $ \mathbf{U} = (H, u, v)^\top $ where
$H,u,$ and $v$ denote, respectively, water depth, depth-averaged velocity in the x-coordinate, and depth averaged velocity in the y-coordinate.
Additionally, we define a tensor for the flux $ \mathbf{F}(\mathbf{U})$,
 \begin{equation}
    \begin{pmatrix}
        Hu & Hv \\
        Hu^2 + \frac{1}{2}g(H^2 -h_b^2) & Huv \\
        Huv & Hv^2 + \frac{1}{2}g(H^2 - h_b^2)  \label{eqn:conservative_flux1}
    \end{pmatrix}\notag
\end{equation}
and forcing vector $\mathbf{f(\mathbf{U})}$:
\begin{equation}
    \begin{pmatrix}
        0 \\
        -g\zeta\frac{\partial h_B}{\partial x} + \tau_B u H - fvH -C_D\frac{\rho_{\text{air}}}{\rho_{water}}\|\mathbf{w}\|w_x +  \frac{H}{\rho_\text{water}} \frac{\partial P}{\partial x} \\
        -g\zeta\frac{\partial h_B}{\partial y} + \tau_B v H + fuH - C_D\frac{\rho_{\text{air}}}{\rho_{\text{water}}}\|\mathbf{w}\|w_y  + \frac{H}{\rho_\text{water}} \frac{\partial P}{\partial y}  \label{eqn:conservative_force1}
    \end{pmatrix}.\notag
\end{equation}
The water surface elevation relative to a fixed geoid is $ \zeta = H - h_b $, where $ h_b $ represents the bathymetry of the land surface, which remains constant over time. 
The parameter $ g $ denotes the gravitational acceleration constant, $ \tau_B $ represents the bottom friction factor, and $ f $ is the Coriolis parameter. 
The water surface drag coefficient is $ C_D $, while $ \rho_{\text{water}} $ and $ \rho_{\text{air}} $ correspond to the reference densities of water and air, respectively. 
The wind velocity vector is $ \mathbf{w} = (w_x, w_y)^\top $, and the atmospheric pressure gradients are given by $ \frac{\partial P}{\partial x} $ and $ \frac{\partial P}{\partial y} $. The explicit expressions for the bottom friction components are $ \frac{\tau_{b_x}}{\rho_0}=\frac{K_{\text{slip}} Q_x}{H} $ and $ \frac{\tau_{b_y}}{\rho_0}=\frac{K_{\text{slip}} Q_y}{H} $. 
The coefficient $ K_{\text{slip}}=c_f|\mathbf{u}| $ follows a quadratic drag law, where
\begin{equation}
c_f=\frac{g n^2}{H^{1 / 3}}     \label{eq:mann_fric}
\end{equation}
depends on Manning’s $ n $ coefficient. 
Since Manning’s $ n $ values vary spatially, they are defined nodewise within the discretized physical domain, forming a piecewise linear representation of the continuous bottom friction field.

We adapt a benchmark test case originally introduced by~\cite{Balzano1998} and later refined in~\cite{Karna2011, Dawson2024}. 
The computational domain is a two-dimensional rectangular basin measuring $13{,}800\,\mathrm{m}$ in length and $7,200,\mathrm{m}$ in width. 
The bathymetry is defined as a uniformly sloping beach, decreasing linearly from $5\ \mathrm{m}$ depth at the offshore boundary to sea level at the shoreline. 
Bottom friction follows Manning’s formulation with a uniform coefficient of $0.02 \ \mathrm{s/m}^{1/3}$. 
While the original benchmark was formulated as a one-dimensional problem, we extend it here to two spatial dimensions to capture realistic cross-shore variability. 
The domain is discretized into 144 uniform triangular elements, each approximately $1,150\ \mathrm{m} \times 1,200\ \mathrm{m}$ in size, as shown in Fig.~\ref{fig:station_locs}, and a Discontinuous Galerkin (DG) finite element method is adapted to handle wetting and drying phenomena following~\cite{Dawson2024}. The initial conditions consist of a flat free surface aligned with the geoid, $\zeta(x, y, 0) = 0\ \mathrm{m}$, and a stationary velocity field: $u(x, y, 0) = v(x, y, 0) = 0\ \mathrm{m/s}$.
At the left boundary ($x = 0$), we impose a time-harmonic open boundary condition with amplitude $2\ \mathrm{m}$ and period 12 hours. 
\begin{table*}
\centering
\begin{tabular}{l|>{\columncolor{gray!10}}c>{\columncolor{gray!10}}c|>{\columncolor{blue!5}}c>{\columncolor{blue!5}}c|>{\columncolor{green!10}}c>{\columncolor{green!10}}c}
\toprule
 & \multicolumn{2}{c|}{\textbf{4D-Var}} & \multicolumn{2}{c|}{\textbf{DC-4DVar}} & \multicolumn{2}{c}{\textbf{DC-WME 4D-Var}} \\
 \cmidrule(lr){2-3} \cmidrule(lr){4-5} \cmidrule(l){6-7}
 & \textbf{RMSE} & \textbf{Data Misfit} & \textbf{RMSE} & \textbf{Data Misfit} & \textbf{RMSE} & \textbf{Data Misfit} \\
\midrule
\textbf{Day 1} & 0.1262 & 0.0397 & 0.1262 & 0.0398 & 0.1262 & 0.0397 \\
\textbf{Day 2} & 0.1564 & 0.0472 & 0.1564 & 0.0472 & 0.1564 & 0.0472 \\
\textbf{Day 3} & 0.1618 & 0.0496 & 0.1617 & 0.0496 & 0.1617 & 0.0495 \\
\textbf{Day 4} & 0.1625 & 0.0499 & 0.1625 & 0.0499 & 0.1624 & 0.0499 \\
\textbf{Day 5} & 0.1627 & 0.0500 & 0.1626 & 0.0500 & 0.1625 & 0.0500 \\
\textbf{Day 6} & 0.1626 & 0.0499 & 0.1626 & 0.0499 & 0.1625 & 0.0499 \\
\textbf{Day 7} & 0.1647 & 0.0596 & 0.1647 & 0.0596 & 0.1647 & 0.0596 \\
\midrule
\textbf{Total}  & 0.2528 & 0.1100 & 0.2528 & 0.1100 & 0.2528 & 0.1100 \\
\bottomrule
\end{tabular}
\caption{Daily and total values of RMSE and data misfit for configurations (A) of the SWE test case under standard 4D-Var, DC 4D-Var, and DC-WME 4D-Var. Results confirm that all three formulations yield nearly identical performance across the 7-day assimilation window. Configuration (B) produced similar results and is omitted for brevity. These outcomes serve as a consistency check, verifying that each method reproduces the expected behavior in a predictable setting.}
\label{tab:error_results}
\end{table*}

\begin{figure*}[t]
\centerline{\includegraphics[width=\textwidth]{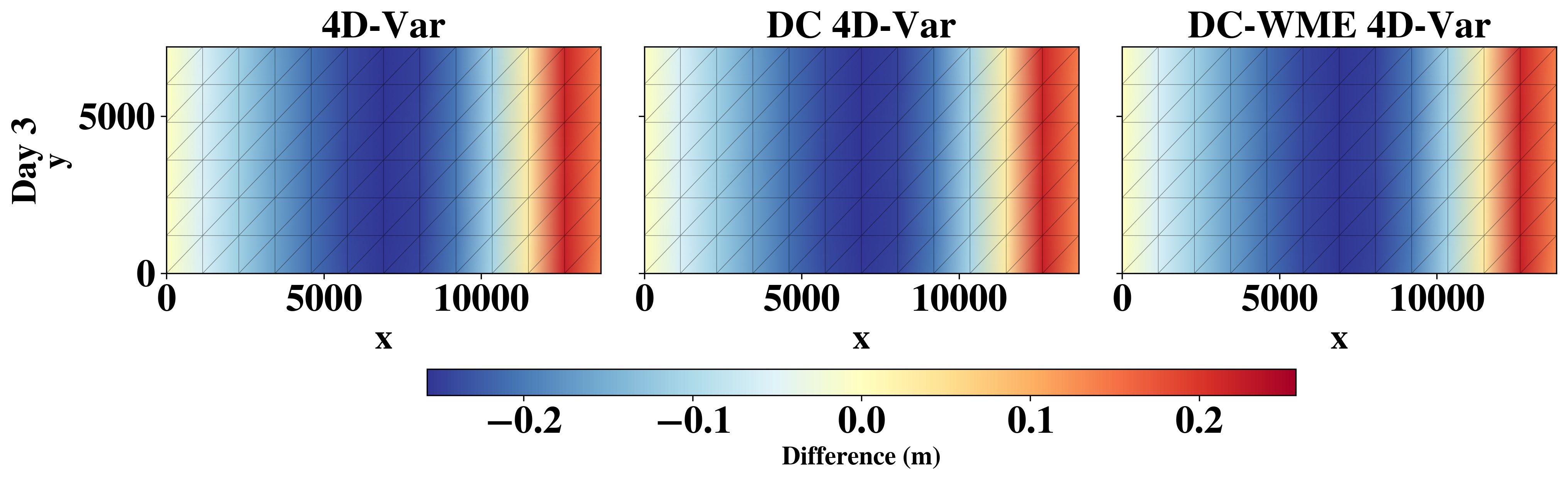}}
\caption{Reconstruction errors in surface elevation on Day 3 of the SWE experiment for configurations (A) and (B). Each column corresponds to one of the three assimilation methods: standard 4D-Var, DC 4D-Var, and DC-WME 4D-Var. Error magnitudes and spatial patterns are nearly indistinguishable across methods, confirming that all implementations behave equivalently in this smooth, periodic regime. This representative snapshot serves as a baseline “sanity check” before applying the methods to more sensitive dynamical systems.}
\label{fig:beach_results}
\end{figure*} 
Wall boundary conditions are enforced on all remaining sides. Simulations are run for a total duration of 7 days, using the SWEMniCS model with a temporal resolution of 600\,s. Synthetic surface elevation observations are generated at 23 spatial locations in configuration (A) and 44 locations in configuration (B), as shown in Fig.~\ref{fig:station_locs}. Observations are assimilated at hourly intervals, each perturbed with independent Gaussian noise of standard deviation $\sigma\_{\mathrm{obs}} = 1.5$, while the background error covariance is modeled as diagonal with standard deviation $\sigma\_{\mathrm{b}} = 4.0$. To avoid inverse crimes, the bathymetry and bottom friction formulations used in the simulation and in the assimilation model differ, as summarized in Table~\ref{tab:sim_params}, ensuring that the task reflects a realistic model–data mismatch.
\begin{table}
\centering
\renewcommand{\arraystretch}{1.4}
\small
\begin{tabular}{@{} l l l @{}}
\toprule
& \textbf{Data Assimilation} & \textbf{Truth} \\
\midrule
\textbf{Bottom Friction} 
& Manning's N & Linear \\
\textbf{Bathymetry (m)} 
& 5 & 5.3  \\
\bottomrule
\end{tabular}
\caption{Summary of key differences used to avoid inverse crimes between the simulation used to generate the synthetic ground truth and the model configuration used for data assimilation.}
\label{tab:sim_params}
\end{table}
Across all configurations, the three methods (standard 4D-Var, DC 4D-Var, and DC-WME 4D-Var) exhibit stable and comparable performance. Daily error fields and aggregate statistics remain consistent across methods, as reflected in Table~\ref{tab:error_results} and Fig.~\ref{fig:beach_results}. Presenting both configurations highlights that the results are insensitive to sensor density: in either case, all three formulations yield nearly identical RMSE and data misfit values. Configuration (A) serves as a baseline where the forcing and boundary conditions produce smooth, periodic dynamics on a rectangular domain. In this setting, all methods recover the true state with consistent accuracy, confirming the correctness of the adjoint implementations and optimization routines on a computationally demanding yet predictably behaved model. Configuration (B), with its denser sensor network, further reinforces these conclusions, showing no meaningful deviation in performance across methods despite the increased observational information.

These SWE experiments serve as a verification step. They confirm that the DC-based 4D-Var formulations function as intended, scale effectively to PDE-based models with nonlinear wetting and drying, and reproduce results that agree with standard 4D-Var when the underlying dynamics favor smooth and well-constrained state estimation.

%\\\\\\\\\\\\\\\\\\\\\\\\\\\\\\\\\\\\\\\\\\\\\\\\\\\\\\\\\\\\\\\\\\\\\\\\\\\\\\\\\\\\\\\\\\\\\\\\\\\\\\\\\\\\\\\\\\\\\\\\\\\\\\\\\\\\\\\\\\\\\\\\\\\\\\\\\\\\\\\\\\\\\\\\\\\
\subsection{The Lorenz 63 System}
The Lorenz (1963) model~\cite{Lorenz1963} (referred to as ''the Lorenz-63 model'') consists of a coupled system of three nonlinear ordinary differential equations:
\begin{equation}
\frac{\partial z_{1}}{\partial t}=\sigma(z_{2}-z_{1}), \frac{\partial z_{2}}{\partial t}=\rho z_{1}-z_{2}-z_{1} z_{3}, \frac{\partial z_{3}}{\partial t}=z_{1} z_{2}-\beta z_{3}.\label{eq:lorenz63}
\end{equation}
The dependent variables are $ z_{1}(t) $, $ z_{2}(t) $, and $ z_{3}(t)$, with common parameter values set to $ \sigma = 10 $, $ \rho = 28 $, and $ \beta = 8/3 $. 
The initial state $ \mathbf{z}_{0} $, which serves as the first background state vector $ \mathbf{z}_0^b $, is drawn from a standard Gaussian prior distribution $ \pi_{\mathrm{init}} \left(\mathbf{z}_{0}, \mathbf{B} \right) $ with a mean of $ \mathbf{z}_{0}$. 
%\\\\\\\\\\\\\\\\\\\\\\\\\\\\\\\\\\\\\\\\\\\\\\\\\\\\\\\\\\\\\\\\\\\\\\\\\\\\\\\\\\\\\\\\\\\\\\\\\\\\\\\\\\\\\\\\\\\\\\\\\\\\\\\\\\\\\\\\\\\\\\\\\\\\\\\\\\\\\\\\\\\\\\\\\\\
\subsubsection{Estimating the Background Variance Bound}
We first use this nonlinear model and its subsequent chaotic trajectories to demonstrate how one can estimate a lower bound on the background variance $ \sigma_b^2 $ required for consistent and stable data assimilation. These results provide empirical support for inflating the background variance in chaotic systems, particularly under sparse or partial observational regimes.

We proceed as follows. First, we consider the theoretical inequality~\eqref{eq:eig_low_bound}, where $ \mathbf{Q}_k = \mathcal{H}_k \circ \mathcal{M}_k $ denotes the composition of the nonlinear model propagator $ \mathcal{M}_k \in \mathbb{R}^{3 \times 3} $ and a fixed linear observation operator $ \mathcal{H}_k \in \mathbb{R}^{2 \times 3} $ that observes the first two components of the state. 
To evaluate this inequality in practice, we propagate the tangent linear model $ \mathbf{M}_k $ of the Lorenz-63 dynamics over a fixed integration time (e.g., $ t = 10.0 $) from a set of randomly sampled initial conditions. 
At the final time of each trajectory, we compute the Jacobian of the flow map $ \mathcal{M}_k $, apply the observation operator $ \mathcal{H}_k $, and construct the matrix $ \mathbf{Q}_k $. 
We then compute the minimum eigenvalue $ \lambda_{\min} \left( \mathbf{Q}_k \mathbf{Q}_k^\top \right) $, which directly enters the lower bound on $ \sigma_b^2 $ via inequality~\eqref{eq:eig_low_bound}. 
We repeat this procedure over 50 independent realizations. For each trajectory, we compute the corresponding bound using fixed values of $ \gamma = 0.1 $, observation noise variance $ \sigma_{\mathrm{obs}}^2 = 4.0 $, and number of observations $ N = 5 $. The resulting distribution of estimated bounds quantifies how the required background variance varies across different regions of the state space. Below, we summarize both the individual estimates and the ensemble mean, which provides a representative lower bound for practical assimilation under the given model and observation configuration. 
\begin{figure}[h]
    \centerline{\includegraphics[width=\textwidth]{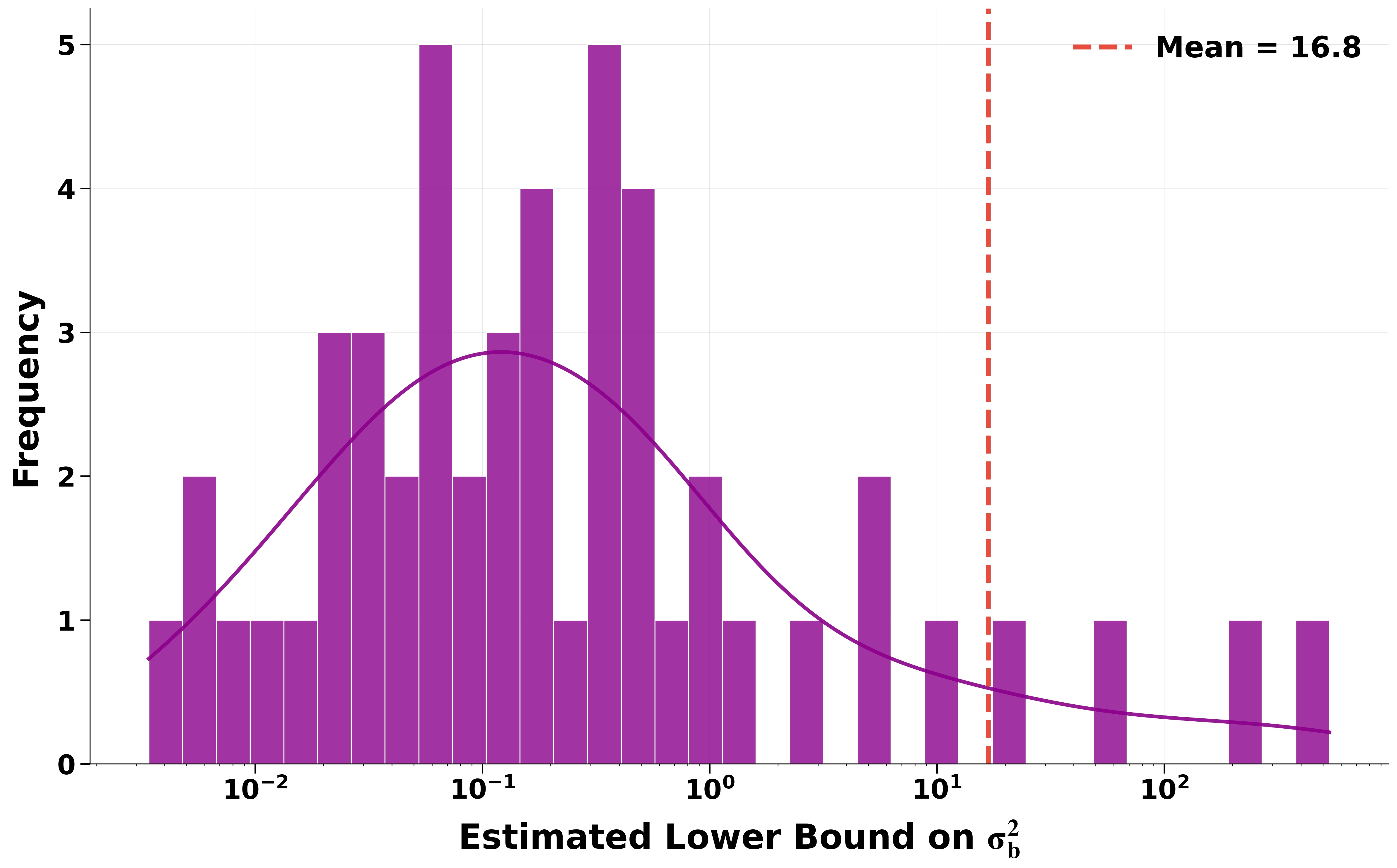}}
    \caption{Distribution of estimated lower bounds on the background variance $ \sigma_b^2 $ computed over 50 random trajectories of the Lorenz-63 system. For each trajectory, the bound is derived from the smallest eigenvalue of the forecast-observation Gram matrix, following the inequality~\eqref{eq:eig_low_bound}. The vertical dashed line denotes the average bound across all samples, which exceeds $ 4\sigma_{\mathrm{obs}}^2 $, confirming the need for background inflation to maintain consistent and stable assimilation in nonlinear regimes.}
    \label{fig:sigma_b_bounds}
\end{figure}
Fig.~\ref{fig:sigma_b_bounds} shows the distribution of estimated lower bounds for the background variance $ \sigma_b^2 $ across 50 random trajectories of the Lorenz 63 system. The mean value of the estimated bounds is approximately $\sigma_b^2 = 16.8$, as indicated by the vertical dashed line, which is over four times the observation noise variance $\sigma_{\mathrm{obs}}^2 = 4.0$. This observation supports the theoretical condition $ \sigma_b^2 \geq 4\sigma_{\mathrm{obs}}^2 $ derived from inequality~\eqref{eq:eig_low_bound}, and it confirms the importance of inflating background variance in systems with limited observability or weak dynamical sensitivity.
\begin{figure}[h]
    \centerline{\includegraphics[width=19pc]{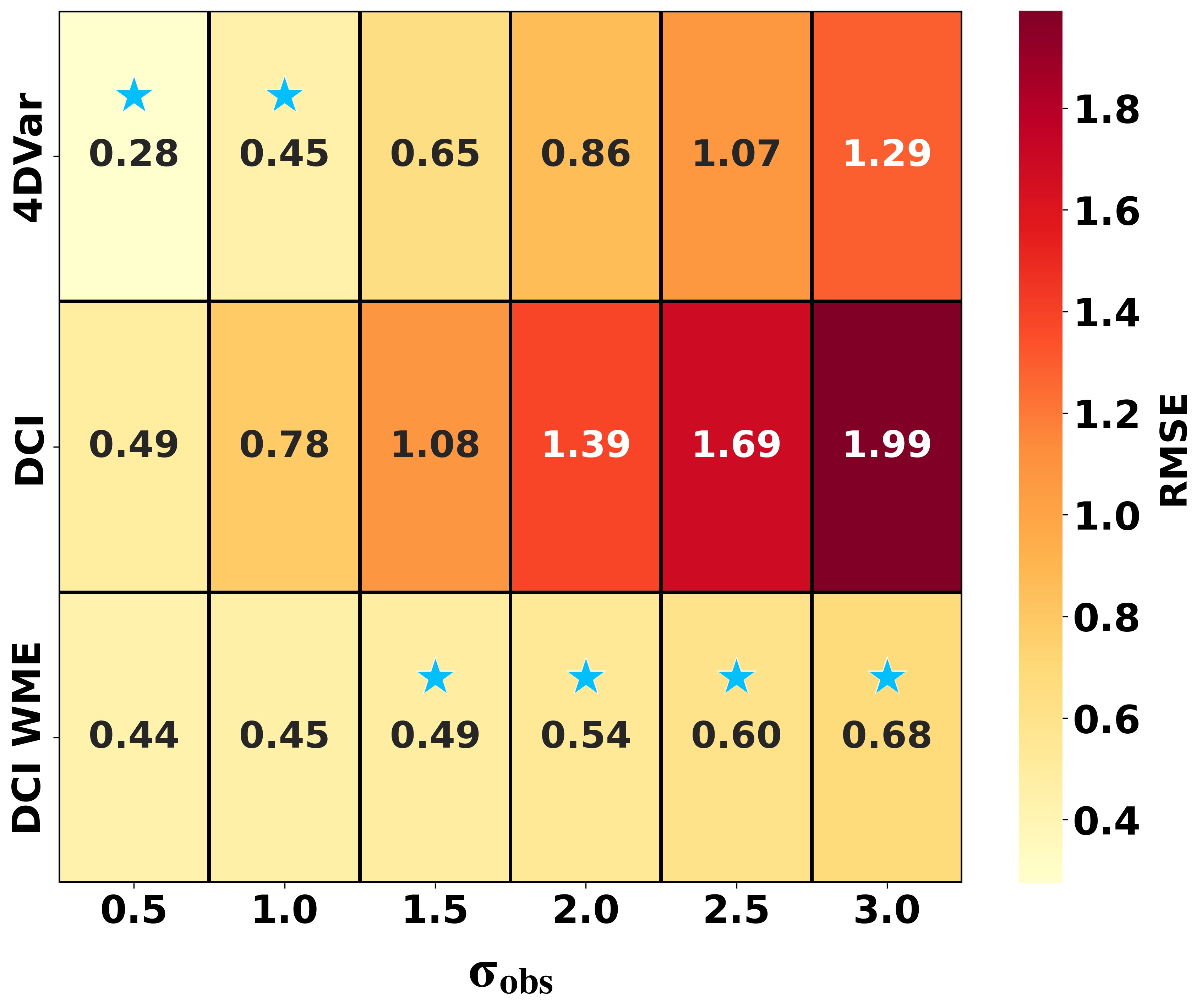}}
    \caption{Time-averaged RMSE for 4D-Var, DC 4D-Var, and DC-WME 4D-Var applied to the Lorenz 63 system under varying observation noise levels. Each method uses a fixed inflation level of $ \alpha = 4\sigma_{\mathrm{obs}}^2 $, corresponding to the theoretically motivated lower bound on background variance. Results demonstrate the robustness of DC-WME 4D-Var to increasing observation noise, with the best-performing method at each noise level indicated by a blue star.}
    \label{fig:l63_inflation}
\end{figure}
To validate the practical implications of this bound, we conduct a complementary experiment analyzing the performance of 4D-Var, DC 4D-Var, and DC-WME 4D-Var under inflation levels fixed at $ \alpha = 4\sigma_{\mathrm{obs}}^2 $ while increasing the observation noise $ \sigma_{\mathrm{obs}} \in \{0.5, 1.0, \dots, 3.0\} $.
Fig.~\ref{fig:l63_inflation} summarizes the results with the lowest-RMSE configuration at each noise level denoted with blue stars. 
These results clearly demonstrate the robustness of DC-WME 4D-Var across all noise levels. 
Unlike standard 4D-Var and DC 4D-Var, which exhibit rapidly increasing RMSE as noise grows, DC-WME 4D-Var maintains low error even under high-noise conditions. 
This advantage becomes especially pronounced when $ \sigma_{\mathrm{obs}} \geq 1.5 $, where DC-WME 4D-Var consistently outperforms the other methods. 

Taken together, the results shown in both Figs.~\ref{fig:sigma_b_bounds} and~\ref{fig:l63_inflation} demonstrate that properly inflating the background variance is not only necessary for theoretical consistency but also yields substantial practical gains in estimation accuracy. 
In particular, the DC-WME 4D-Var framework benefits significantly from predictable inflation strategies, using increased model uncertainty to reduce assimilation error in noisy, nonlinear regimes such as Lorenz 63.

%\\\\\\\\\\\\\\\\\\\\\\\\\\\\\\\\\\\\\\\\\\\\\\\\\\\\\\\\\\\\\\\\\\\\\\\\\\\\\\\\\\\\\\\\\\\\\\\\\\\\\\\\\\\\\\\\\\\\\\\\\\\\\\\\\\\\\\\\\\\\\\\\\\\\\\\\\\\\\\\\\\\\\\\\\\\
\subsubsection{Time Averaged Root Mean Squared Error}

We now demonstrate that DC-WME 4D-Var yields the most accurate and stable estimates throughout the simulation of the Lorenz-63 model.
Figure~\ref{fig:ta_rmse_l63} shows the time series of root mean squared error (RMSE) computed over 1000 model time steps for the Lorenz-63 system of the 4D-Var, DC 4D-Var, and DC-WME 4D-Var methods. 
Each method assimilates observations over 50 assimilation cycles, and we compute the RMSE by comparing the analysis state to the true hidden state at each time step.

\begin{figure}[t]
    \centerline{\includegraphics[width=\textwidth]{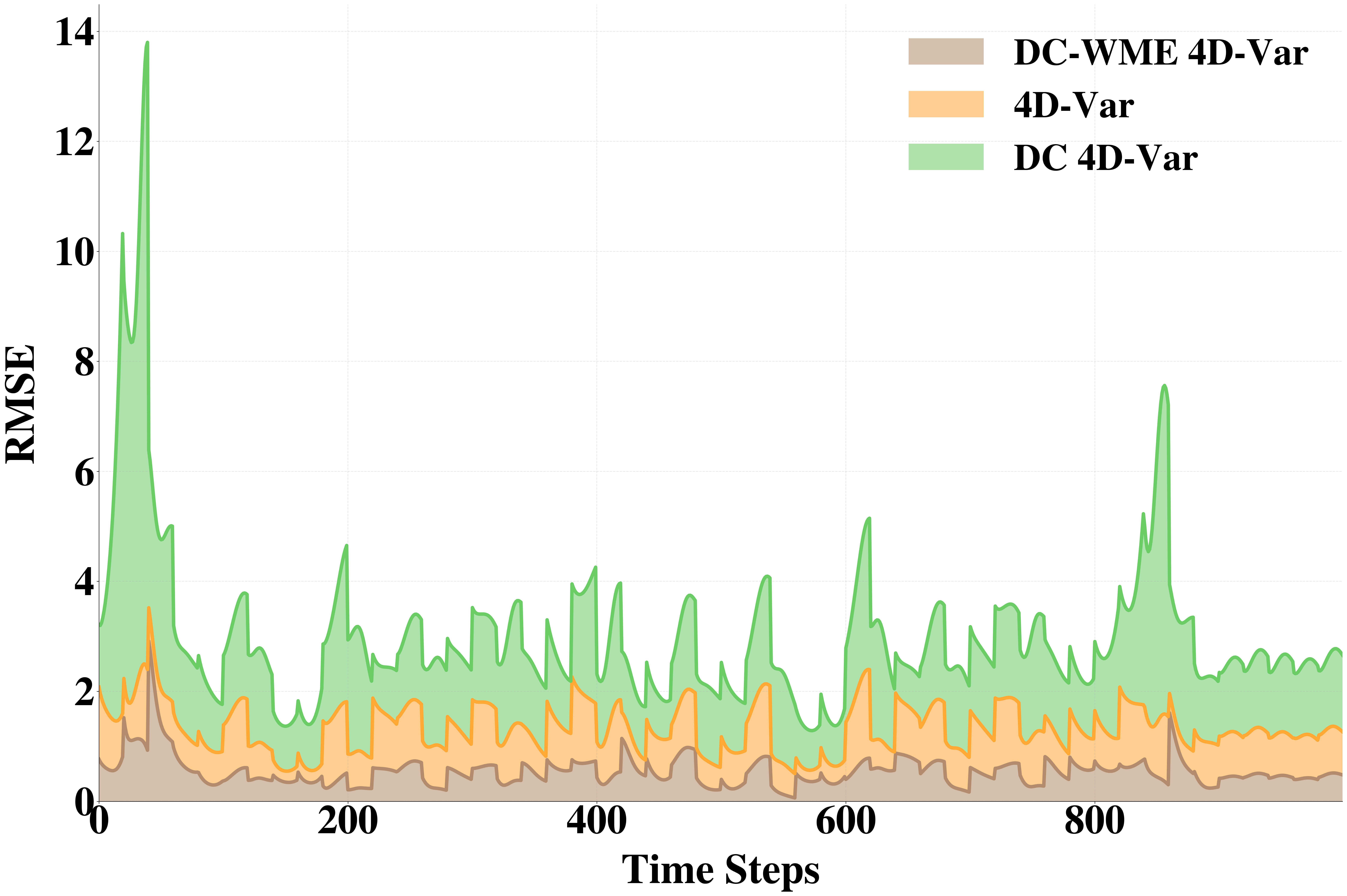}}
    \caption{Time-averaged Root Mean Squared Error (RMSE) over 1000 model time steps for 4D-Var, DC 4D-Var, and DC-WME 4D-Var applied to the Lorenz-63 system. RMSE is computed using the analysis state across 50 assimilation cycles. The results show that DC-WME 4D-Var consistently achieves the lowest error and exhibits greater stability over time, especially in the presence of nonlinear and chaotic dynamics.}
    \label{fig:ta_rmse_l63}
\end{figure}

Clearly, the DC-WME 4D-Var RMSE remains consistently low across the entire time horizon, with minimal variability. 
By contrast, standard 4D-Var produces higher average error and greater temporal fluctuations, while DC 4D-Var, though less accurate than standard 4D-Var, exhibits greater sensitivity to transient dynamics. 

The overall reduction in variation of the DC-WME 4D-Var RMSE is the consequence of computing the weighted mean of model predictions and observations through the WME QoI map in each assimilation window. 
By incorporating predictability-aware weighting into the assimilation process, DC-WME 4D-Var effectively suppresses the influence of poorly constrained directions and emphasizes consistency with the underlying dynamics. 
As a result, it achieves not only lower RMSE but also improved robustness over time, making it a compelling alternative to conventional variational assimilation approaches in nonlinear, low-dimensional systems.

%\\\\\\\\\\\\\\\\\\\\\\\\\\\\\\\\\\\\\\\\\\\\\\\\\\\\\\\\\\\\\\\\\\\\\\\\\\\\\\\\\\\\\\\\\\\\\\\\\\\\\\\\\\\\\\\\\\\\\\\\\\\\\\\\\\\\\\\\\\\\\\\\\\\\\\\\\\\\\\\\\\\\\\\\\\\
\subsection{Lorenz-96}

We now consider the Lorenz-96 model~\cite{Lorenz1996}, which allows us to vary the degrees of freedom (DoF) and demonstrate the performance of the methods across multiple scenarios. 
Lorenz originally developed this system to model the temporal evolution of a generic scalar field, such as temperature, advected through a fluid within a cyclic spatial domain (e.g., along a constant latitude). 
The model represents this process using a spatially discretized set of coupled ordinary differential equations, given by
\begin{equation}
\frac{d \mathbf{z}_{k}}{d t}=\left(\mathbf{z}_{k+1}-\mathbf{z}_{k-2}\right) \mathbf{z}_{k-1}-\mathbf{z}_{k}+F, \quad k=1, \ldots, K.
\end{equation}
Let $ k $ denote one of $ K $ equally spaced grid points along the circular domain. 
Each state asymmetrically interacts with its neighboring states and follows periodic boundary conditions, where $ \mathbf{z}_{K+1} = \mathbf{z}_1 $, $ \mathbf{z}_{0} = \mathbf{z}_{k} $, and $ \mathbf{z}_{-1} = \mathbf{z}_{K-1} $. 
The indices increase in the positive direction and decrease in the negative direction. 
The nonlinear terms represent advection, the linear term captures dissipation, and $ F $ acts as a forcing term. In this study, we set $ F = 8 $, which generates chaotic dynamics.

\begin{figure}[t]
\centerline{\includegraphics[width=\textwidth]{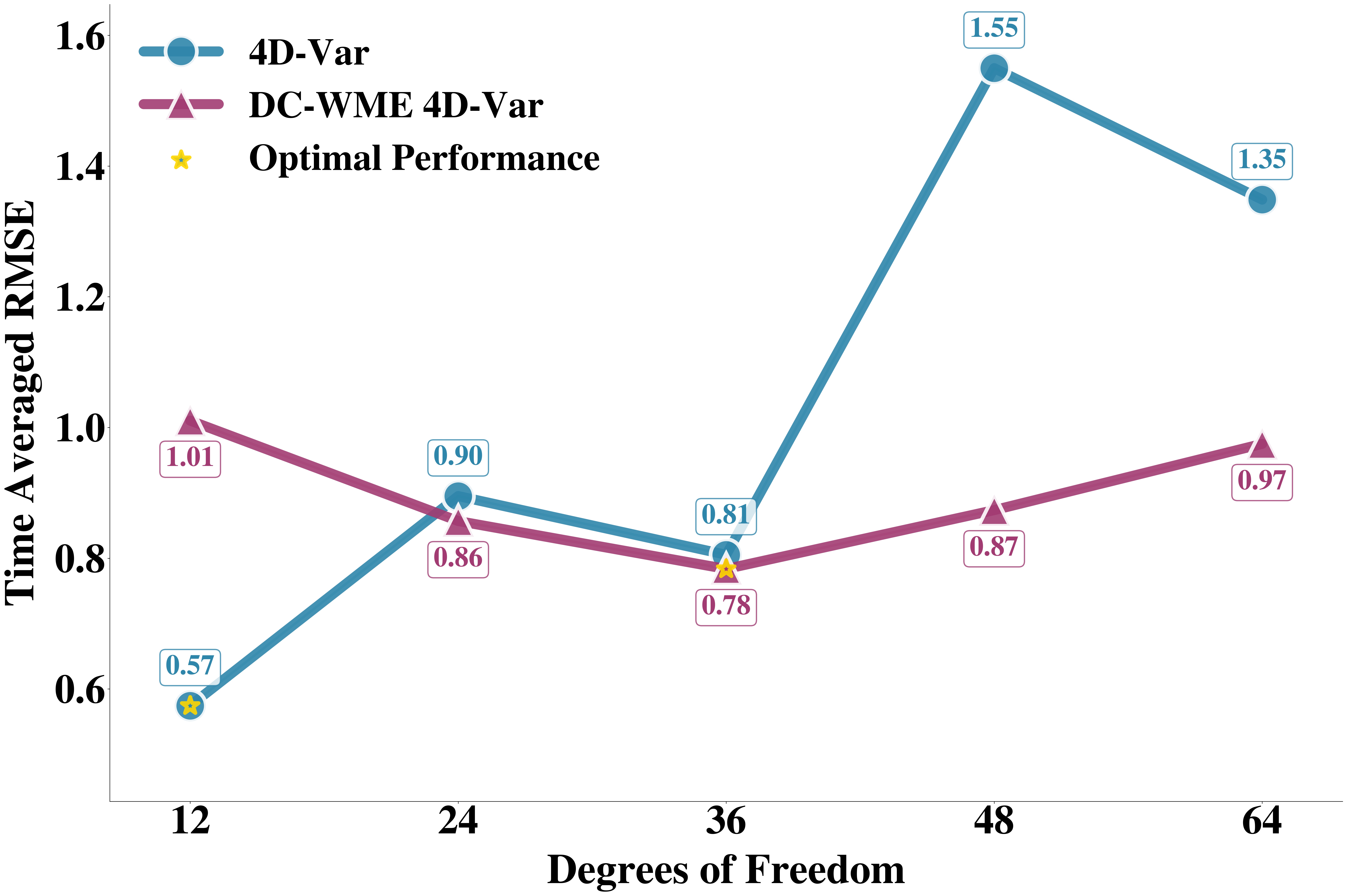}}
    \caption{Time-averaged Root Mean Squared Error (RMSE) for 4D-Var and DC-WME 4D-Var applied to the Lorenz 96 system across increasing numbers of degrees of freedom. As system dimensionality increases, standard 4D-Var suffers from growing estimation error, reaching an RMSE at $ K = 48 $. In contrast, DC-WME 4D-Var maintains low error throughout, achieving its optimal RMSE at $ K = 36 $ and remaining below 1.0 even at $ K = 64 $. These results demonstrate the superior scalability and stability of the DC-WME formulation in high-dimensional, partially observed chaotic systems.}
    \label{fig:rdvar_dofs}
\end{figure}
%\\\\\\\\\\\\\\\\\\\\\\\\\\\\\\\\\\\\\\\\\\\\\\\\\\\\\\\\\\\\\\\\\\\\\\\\\\\\\\\\\\\\\\\\\\\\\\\\\\\\\\\\\\\\\\\\\\\\\\\\\\\\\\\\\\\\\\\\\\\\\\\\\\\\\\\\\\\\\\\\\\\\\\\\\\\

\subsubsection{Sensitivity to Degrees of Freedom}
We integrate the system over 500 time steps, perform data assimilation every 20 steps, and take partial observations every 4 steps on alternating state components. We fix the observational noise with standard deviation $ \sigma_{\text{obs}} = 1.2 $, and we scale the predictive uncertainty to ensure the inequality $ \sigma_{\text{pred}} \ge 4\sigma_{\text{obs}}^2 $ holds.
The time-averaged RMSE for both standard 4DVar and DC-WME 4DVar algorithms is shown in Fig.~\ref{fig:rdvar_dofs} as a function of DoF.
Across most all configurations, DC-WME 4D-Var consistently outperforms standard 4D-Var in terms of time-averaged RMSE with the performance gap increasing as the DoF grows. Standard 4D-Var achieves its best performance for low-dimensional systems (e.g., $ K = 12 $), which we mark as the optimal case. However, its accuracy deteriorates sharply beyond $ K = 36 $,  and remaining high at $ K = 64 $. This degradation reflects growing sensitivity to unobserved modes and accumulating forecast error in higher-dimensional regimes. The DC-WME 4D-Var, however, maintains robust performance across all DoFs. It achieves its best result at $K = 36$, which is also marked as optimal. 
As we increase the dimensionality, DC-WME 4D-Var experiences only a mild increase in RMSE, between $ K = 36 $ and $ K = 64 $. This resilience suggests that the DC-WME 4D-Var framework effectively balances information from both observed and unobserved components, mitigating the dimensional curse that undermines standard 4D-Var. 
Overall, these results demonstrate that DC-WME 4D-Var has the potential to improve scalability and deliver more consistent forecast skill as the complexity of the dynamical system increases. 
Its ability to maintain low RMSE even in high-dimensional regimes further highlights its potential for large-scale operational data assimilation applications.

%\\\\\\\\\\\\\\\\\\\\\\\\\\\\\\\\\\\\\\\\\\\\\\\\\\\\\\\\\\\\\\\\\\\\\\\\\\\\\\\\\\\\\\\\\\\\\\\\\\\\\\\\\\\\\\\\\\\\\\\\\\\\\\\\\\\\\\\\\\\\\\\\\\\\\\\\\\\\\\\\\\\\\\\\\\\
\subsubsection{Forecasting and Model Bias}

The forecasting results of the various 4D-Var algorithms on the Lorenz-96 system over 1000 model time steps are presented in Fig.~\ref{fig:rmse_bias}. 
Let the bias time series associated with a given model be denoted by the vector
\begin{equation}
    \mathbf{b} = [b_1, b_2, \dots, b_T]^\top \in \mathbb{R}^{T},
\end{equation}
where $ b_t $ represents the signed bias at time step $ t $. To facilitate consistent visualization across models and time intervals, we normalize this series to the interval $ [-1, 1] $ using the infinity norm:
\begin{equation}
    \tilde{\mathbf{b}} = \frac{\mathbf{b}}{\|\mathbf{b}\|_{\infty}} = \left[ \frac{b_1}{\|\mathbf{b}\|_{\infty}}, \frac{b_2}{\|\mathbf{b}\|_{\infty}}, \dots, \frac{b_T}{\|\mathbf{b}\|_{\infty}} \right],
\end{equation}
where $\|\mathbf{b}\|_{\infty} := \max_{1 \leq t \leq T} |b_t|$ denotes the maximum absolute bias over time. This normalization preserves the relative structure of the bias signal while ensuring that its largest magnitude equals one. To retain interpretability of the original scale, the normalization factor $ \|\mathbf{b}\|_{\infty} $ is explicitly annotated within each subplot.
\begin{figure*}[h!]
\centerline{\includegraphics[width=\textwidth]{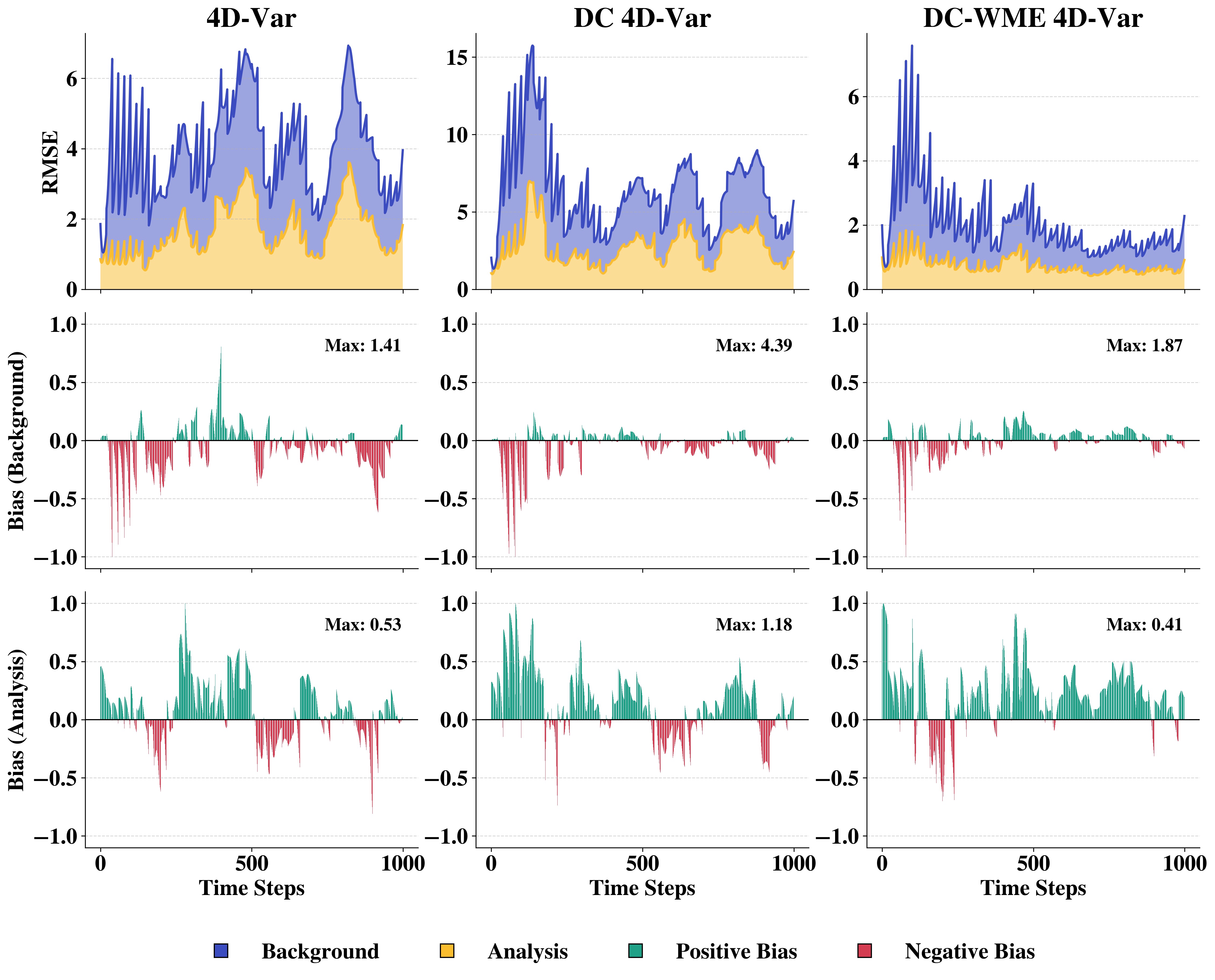}}
\caption{Forecasting results for 4D-Var, DC 4D-Var, and DC-WME 4D-Var applied to the Lorenz 96 system over 1000 model time steps. Each subfigure plots RMSE (top row) and normalized bias (middle and bottom rows) for the background (blue) and analysis (yellow) fields. Across all methods, analysis fields consistently reduce both RMSE and bias relative to their background counterparts. DC-WME 4D-Var achieves the lowest RMSE and bias magnitudes overall, with its analysis bias peaking at only 0.41 compared to 0.71 for DC 4D-Var and 0.95 for standard 4D-Var. These results underscore the effectiveness of the WME formulation in improving both accuracy and stability over time.}
\label{fig:rmse_bias} 
\end{figure*}

The plots in Fig.~\ref{fig:rmse_bias} highlight several key differences in forecast skill and error structure across the methods. 
The top row illustrates that DC-WME 4D-Var consistently achieves the lowest RMSE in both background and analysis fields throughout the assimilation window. 
The shaded area between the background and analysis curves remains smallest for DC-WME 4D-Var, indicating tighter alignment between prior estimates and the assimilated solution. 
In the standard 4D-Var and DC 4D-Var cases, the background RMSE exhibits periodic fluctuations and sharp peaks while DC-WME 4D-Var maintains a more stable and suppressed RMSE trajectory.
This suggests an improved robustness to transient model instabilities and enhanced temporal consistency is achieved by the DC-WME 4D-Var method. 
The normalized bias plots (middle and bottom rows) further illustrate the error-reducing effect of assimilation. 
For all methods, the analysis bias remains consistently closer to zero than the background bias.

However, DC-WME 4D-Var achieves the most dramatic bias reduction, with the lowest maximum bias magnitudes in the analysis fields ($ \max = 0.41 $). 
This result reflects the enhanced ability of the WME formulation to mitigate systematic over and underestimation over time. 
Additionally, DC-WME 4D-Var produces a more symmetric distribution of positive and negative biases, whereas standard 4D-Var and DC 4D-Var show more persistent negative bias trends in the background estimates. 
This symmetry in DC-WME 4D-Var likely reflects a more balanced correction mechanism enabled by its QoI-weighted structure.

%\\\\\\\\\\\\\\\\\\\\\\\\\\\\\\\\\\\\\\\\\\\\\\\\\\\\\\\\\\\\\\\\\\\\\\\\\\\\\\\\\\\\\\\\\\\\\\\\\\\\\\\\\\\\\\\\\\\\\\\\\\\\\\\\\\\\\\\\\\\\\\\\\\\\\\\\\\\\\\\\\\\\\\\\\\\
\begin{figure*}[t]
\centerline{\includegraphics[width=\textwidth]{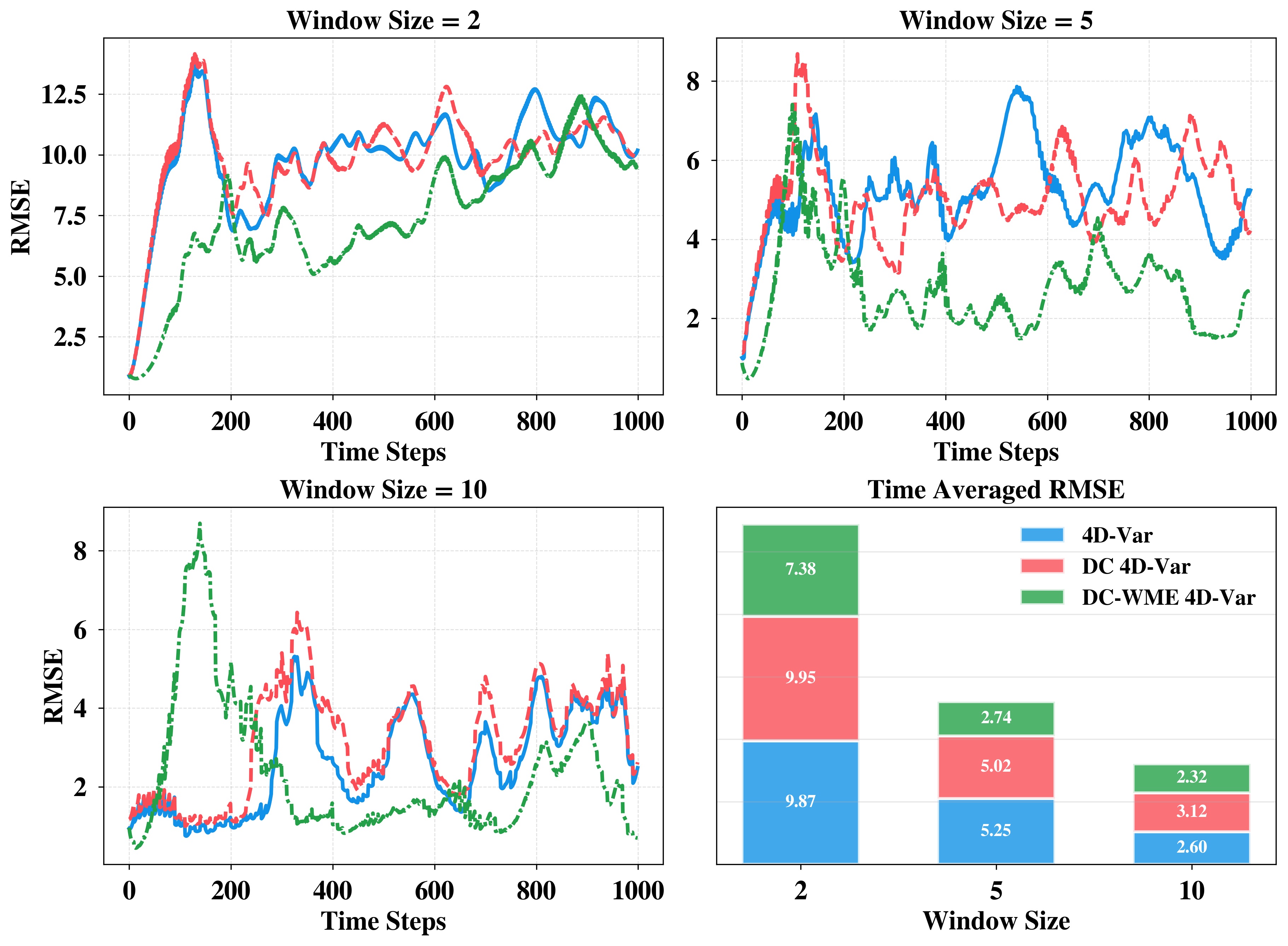}}
    \caption{Time series of Root Mean Squared Error (RMSE) for 4D-Var, DC 4D-Var, and DC-WME 4D-Var applied to the Lorenz 96 system under varying assimilation window lengths. Each subplot corresponds to a different window size: 2-step, 5-step, and 10-step. Across all configurations, DC-WME 4D-Var achieves the lowest RMSE, with the most pronounced improvement observed in the 2-step window case, where its average RMSE is 7.38 compared to 9.87 for standard 4D-Var. These results highlight the robustness of the DC-WME formulation under reduced observation intervals and limited temporal assimilation depth.}
    \label{fig:window_size_plot_l96}
\end{figure*}

\subsubsection{Effect of Window Size}

We now assess the impact of assimilating fewer observations by progressively shortening the assimilation window size to 10 steps, 5 steps, and 2 steps. 
We observe the Lorenz-96 system at every other DoF. 
The RMSE of the analysis fields shown in Fig.~\ref{fig:window_size_plot_l96} demonstrate the performance of each method for each window size. 
In all cases, DC-WME 4D-Var achieves the lowest RMSE, with the improvement most pronounced for shorter windows. 

These results indicate that the DC-WME 4D-Var formulation becomes substantially more robust when we assimilate observations more frequently but in smaller batches. For a window size of 5, DC-WME 4D-Var maintains superior performance, with a substantially lower average RMSE compared to both standard 4D-Var and DC 4D-Var. Even when the window size is extended to 10 steps, where traditional methods typically perform better due to a longer observational lookahead, DC-WME 4D-Var still yields the lowest RMSE. The stacked bar chart of time-averaged RMSEs clearly supports this conclusion, showing consistent and significant improvement of DC-WME 4D-Var across all window sizes considered. The RMSE trajectories of standard 4D-Var and DC 4D-Var exhibit larger peaks and greater variability where the DC-WME 4D-Var consistently produces lower and more stable RMSE values over time, particularly during high-error episodes. This stability underscores the benefit of incorporating predictability-based weighting into the cost function.  Taken together, these results demonstrate that DC-WME 4D-Var can both improve forecast accuracy while simultaneously enhancing resilience to the challenges posed by smaller assimilation windows, a scenario that mimics limited observational availability in real-world systems.
\begin{figure*}[t]
    \centerline{\includegraphics[width=\textwidth]{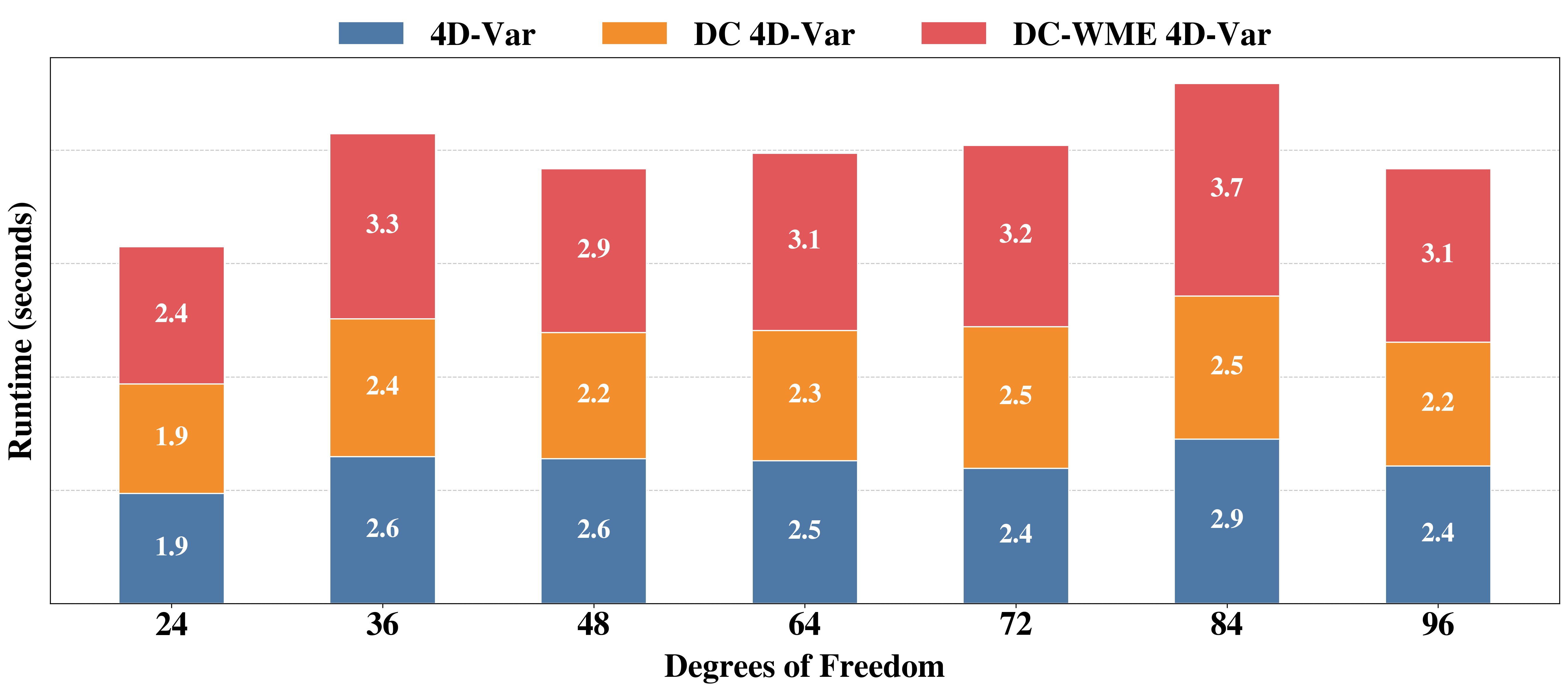}}
    \caption{Wall-clock runtime for 4D-Var, DC 4D-Var, and DC-WME 4D-Var applied to the Lorenz-96 system over a simulation of length $ t = 20000 $, evaluated across increasing DoF. 
    For low-dimensional cases (e.g., 24 DoF), DC-WME 4D-Var exhibits the highest runtime due to fixed overhead associated with computing the weighted mean error map. 
    As system size increases, the relative cost difference remains stable: at 96 DoF, DC-WME 4D-Var completes within 0.7 seconds of standard 4D-Var.
    This demonstrates that the added computational cost of DC-WME remains moderate and scales reasonably with system size, making it a viable option even for high-dimensional settings.}
    \label{fig:comp_cost_plot_l96}
\end{figure*}

%\\\\\\\\\\\\\\\\\\\\\\\\\\\\\\\\\\\\\\\\\\\\\\\\\\\\\\\\\\\\\\\\\\\\\\\\\\\\\\\\\\\\\\\\\\\\\\\\\\\\\\\\\\\\\\\\\\\\\\\\\\\\\\\\\\\\\\\\\\\\\\\\\\\\\\\\\\\\\\\\\\\\\\\\\\\

\subsubsection{Computational Costs}

Recall that the analysis of Section~\ref{sec:comp_considerations}.\ref{subsec:comp_cost} suggests that all the 4D-Var methods should incur a similar computational cost. 
This is confirmed in Fig.~\ref{fig:comp_cost_plot_l96}, which shows the average computational cost, measured in wall-clock runtime over five data assimilation experiments, for each of the three 4D-Var algorithms applied to the Lorenz-96 system over a time horizon of $ t = 20{,}000 $, across a range of problem sizes characterized by the number of DoF. 
The data are assimilated every 10 time steps; therefore, each algorithm must execute 2,000 assimilation windows.

Among the three approaches, the baseline 4D-Var algorithm consistently exhibits the lower run times across all problem sizes as expected. The plot reveals several key trends.  First, the baseline 4D-Var algorithm (blue bars) exhibits modest scaling with increasing DoF. Runtime grows from 1.9 seconds at $K=24$ to 2.9 seconds at $K=84$, before slightly decreasing to 2.4 seconds at $K=96$. 
This relative flatness suggests that the core adjoint computations in traditional 4D-Var scale well and may benefit from memory locality or caching at higher DoF. The DC 4D-Var algorithm (orange bars) adds a small but consistent overhead to the baseline runtime across all DoF (approximately 2.2–2.5 seconds), indicating that the additional consistency terms are computationally lightweight and scale linearly with problem size. 

The DC-WME 4D-Var algorithm (red bars) incurs the highest cost, particularly as the DoF increases. 
This growth reflects the cost of evaluating the WME QoI and its gradient, which scales with both model complexity and window length. 
The peak observed at $K=84$ may be due to increased numerical instability or memory pressure during these computations. 
Despite this, the runtime gap between methods remains constant as the problem size grows. 
At 96 DoF, the overhead of DC-WME 4D-Var relative to standard 4D-Var is only about 0.7 seconds. 
This observation suggests that for larger-scale systems, the relative cost increase is marginal compared to the total optimization cost and appears to be a reasonable trade-off to achieve the superior accuracy and stability demonstrated in the other experiments. 

%\\\\\\\\\\\\\\\\\\\\\\\\\\\\\\\\\\\\\\\\\\\\\\\\\\\\\\\\\\\\\\\\\\\\\\\\\\\\\\\\\\\\\\\\\\\\\\\\\\\\\\\\\\\\\\\\\\\\\\\\\\\\\\\\\\\\\\\\\\\\\\\\\\\\\\\\\\\\\\\\\\\\\\\\\\\
\section{Summary and Conclusions}

This work introduces and analyzes a Weighted Mean Error (WME) extension of a Data-Consistent 4D-Var framework to improve variational data assimilation in nonlinear, partially observed, and high-dimensional systems. 
By incorporating a predictability-aware weighting of the model–data misfit, the proposed DC-WME 4D-Var method integrates structural insights from data-consistent inversion theory into the variational assimilation framework. 
The implementation leverages automatic differentiation and quasi-Newton optimization within a differentiable programming environment, ensuring flexibility and scalability. 
Numerical experiments on the Lorenz-63 system confirm that inflation of the background variance is essential in chaotic regimes. 
Empirical estimates of a lower bound on $\sigma_b^2$ indicate that sufficient prior uncertainty is necessary for stable and consistent assimilation, particularly under weak observability.

These results support theoretical expectations and validate variance inflation as a safeguard against numerical instability and overconfident estimation.

In the shallow water equations (SWE) benchmark test, DC-WME 4D-Var demonstrates practical effectiveness in PDE-constrained settings. 
Across the benchmark test cases, DC-WME 4D-Var consistently reduces root mean squared error (RMSE) and forecast bias relative to both standard 4D-Var and unweighted DC-4DVar. 
The method maintains robust performance under increased observation noise and growing state dimensionality. 
Notably, it retains accuracy under short assimilation windows, where baseline methods often degrade. 
This robustness arises from the WME map’s ability to emphasize well-observed dynamical directions, resulting in more accurate state reconstructions. 

Despite moderate computational overhead, the DC-WME method scales well with system size. 
The cost of computing and differentiating the WME term is offset by improved forecast skill and numerical stability, making the method a practical choice for high-dimensional data assimilation problems. 
This study shows that integrating predictability-aware weighting into the 4D-Var framework substantially improves estimation accuracy, temporal consistency, and robustness. 
The DC-WME approach strengthens the link between statistical and variational formulations of inverse problems and provides a scalable foundation for modern data assimilation. 
Future work will extend the DC-WME 4D-Var methodology to large-scale storm surge forecasting. 
In these applications, the SWE govern coastal inundation dynamics, and the ability to incorporate spatially varying structure and uncertainty makes the method a promising candidate for real-time, high-fidelity storm surge prediction.

%%%%%%%%%%%%%%%%%%%%%%%%%%%%%%%%%%%%%%%%%%%%%%%%%%%%%%%%%%%%%%%%%%%%%
% CREDIT STATEMENT
%%%%%%%%%%%%%%%%%%%%%%%%%%%%%%%%%%%%%%%%%%%%%%%%%%%%%%%%%%%%%%%%%%%%%
\section{CRediT authorship contribution statement}
\textbf{Rylan Spence:} Writing – original draft, Visualization, Validation, Software, Methodology, Investigation, Formal analysis, Conceptualization;
\textbf{Troy Butler:} Writing – original draft, review \& editing, Supervision, Project administration, Conceptualization, Funding acquisition. 
\textbf{Clint Dawson: }  Writing - review \& editing, Supervision, Conceptualization, Funding acquisition.
%%%%%%%%%%%%%%%%%%%%%%%%%%%%%%%%%%%%%%%%%%%%%%%%%%%%%%%%%%%%%%%%%%%%%
% DATA AVAILABILITY STATEMENT
%%%%%%%%%%%%%%%%%%%%%%%%%%%%%%%%%%%%%%%%%%%%%%%%%%%%%%%%%%%%%%%%%%%%%
\section{Data Availability}
The code to recreate all of the examples in this work are located at \href{https://github.com/rspence821505/Variational-Data-Consistent-Assimilation}{Variational Data Consistent Assimilation Github}. The SWEMniCSx library \href{https://github.com/UT-CHG/SWEMniCS}{SWEMniCS Github} can serve as a practical guide for practitioners to learn and start applying shallow water equations to coastal engineering problems. 

%%%%%%%%%%%%%%%%%%%%%%%%%%%%%%%%%%%%%%%%%%%%%%%%%%%%%%%%%%%%%%%%%%%%%
% Competing Interest STATEMENT
%%%%%%%%%%%%%%%%%%%%%%%%%%%%%%%%%%%%%%%%%%%%%%%%%%%%%%%%%%%%%%%%%%%%%
\section{Declaration of competing Interest}
The authors declare that they have no known competing financial interests or personal relationships that could have appeared to
influence the work reported in this paper.
%%%%%%%%%%%%%%%%%%%%%%%%%%%%%%%%%%%%%%%%%%%%%%%%%%%%%%%%%%%%%%%%%%%%%
% ACKNOWLEDGMENTS
%%%%%%%%%%%%%%%%%%%%%%%%%%%%%%%%%%%%%%%%%%%%%%%%%%%%%%%%%%%%%%%%%%%%%
\section{Acknowledgments}

T.~Butler's work is supported by the National Science Foundation under Grant No.~DMS-2208460 and is also supported by NSF IR/D program while working at National Science Foundation. 
C.~Dawson’s and R.~Spence's work is supported in part by the National Science Foundation No. DMS-2208461. 
However, any opinion, finding, conclusions, or recommendations expressed in this material are those of the authors and do not necessarily reflect the views of the National Science Foundation. Finally, the authors would like to thank Dr. Benjamin Pachev and Dr. Mark Loveland for their invaluable insights in adapting the SWEMniCSx modeling software into a data assimilation framework.

%%%%%%%%%%%%%%%%%
%APPENDIXES
%%%%%%%%%%%%%%%%%

\appendix

\section{Data-Consistent Inversion Summary}\label{app:DCI_summary}

We provide a high-level summary of the Data-Consistent Inversion (DCI) problem and its solution as covered in this work and refer the interested reader to~\cite{Butler2018} for more details.

\subsection{Consistency}\label{app:consistency}
Given a probability measure $\mathbb{P}_\mathcal{D}$ on $(\mathcal{D}, \mathcal{B}_\mathcal{D}, \mu_\mathcal{D})$ with density $\pi_\mathcal{D}$, we say a probability measure $\mathbb{P}_{Z}$ on $(Z, \mathcal{B}_{\mathcal{Z}}, \mu_Z)$, with density $\pi_Z$, is consistent with respect to the measurable map $Q:Z\to\mathcal{D}$ if for all $E\in\mathcal{B}_\mathcal{D}$, 
\begin{equation}
\begin{split}
    \mathbb{P}_Z(Q^{-1}(E)) 
    &= \int_{Q^{-1}(E)} \pi_Z(z) \, d\mu_Z \\
    &= \int_{E} \pi_{\mathcal{D}} \!\left(Q(\mathbf{z}_0)\right) d\mu_\mathcal{D} =\mathbb{P}_\mathcal{D}(E).   \label{eq:consistency}
\end{split}
\end{equation}
%\\\\\\\\\\\\\\\\\\\\\\\\\\\\\\\\\\\\\\\\\\\\\\\\\\\\\\\\\\\\\\\\\\\\\\\\\\\\\\\\\\\\\\\\\\\\\\\\\\\\\\\\\\\\\\\\\\\\\\\\\\\\\\\\\\\\\\\\\\\\\\\\\\\\\\\\\\\\\\\\\\\\\\\\\\\
\subsection{Stochastic Inverse Problem}\label{subsec:sip}

Given an observed probability measure, $\mathbb{P}_{\mathrm{obs}}$ on $\mathcal{D}$, a solution to the stochastic inverse problem is a probability measure $\mathbb{P}_{\mathcal{Z}}$ on $\mathcal{Z}$ that is consistent in the sense of definition~\eqref{eq:consistency}. 
Assuming $\mathbb{P}_{\mathcal{Z}}$ and $\mathbb{P}_{\mathrm{obs}}$ admit probability densities, $\pi_{\mathcal{Z}}$ and $\pi_{\mathrm{obs}}$, this is equivalent to:
\begin{equation}
\int_{Q^{-1}(E)} \pi_{Z}(\mathbf{z}_0) \mu_{\mathcal{Z}}=\int_E \pi_{\mathrm{obs}}(y) \mu_{\mathcal{D}} .
\end{equation}
When the mapping $Q$ is injective and its Jacobian is clearly defined, the stochastic inverse problem yields a unique solution that, in theory, can be calculated through the standard change of variables formula:
\begin{equation}
\pi_{\mathrm{up}}(\mathbf{z}_{0})=\pi_{\mathrm{obs}}(Q(\mathbf{z}_{0}))\left|J_Q\right|
\end{equation}
where $\left|J_Q\right|$ is the determinant of the Jacobian of $Q(\mathbf{z}_{0})$. 
The main challenge in solving the stochastic inverse problem stems from cases where $Q^{-1}(y)$ is ill-defined, i.e., multiple $\mathbf{z}_0$ map to the same $y \in \mathcal{D}$. 
This can occur either from nonlinearity in $Q$ or differences in dimension between the spaces.
No matter the cause, it is often the case that direct application of the change of variables formula is not possible. 

\subsection{Density-Based Solution}\label{app:DCI_solution}

The work of \cite{Butler2018} demonstrated that by adopting a suitable initial probability density, $\pi_{\mathrm{init}}$, which satisfies a predictability assumption, it is possible to obtain a unique solution in the form of a multiplicative update to this density.
\begin{theorem}[Existence and Uniqueness]\label{thm:sip_exist_unique}
The probability measure $\mathbb{P}_{\mathrm{up}}$ on $\left(\mathcal{Z}, \mathcal{B}_{\mathcal{Z}}\right)$ defined by
\begin{equation}
\mathbb{P}_{\mathrm{up}}(A)=\int_{\mathcal{D}}\left(\int_{A \cap Q^{-1}(y)} \pi_{\mathrm{init}}(\mathbf{z}_{0}) \frac{\pi_{\mathrm{obs}}(Q(\mathbf{z}_{0}))}{\pi_{\mathrm{pred}}(Q(\mathbf{z}_{0}))} d \mu_{\mathcal{Z}, y}\right) d \mu_{\mathcal{D}}, \label{eq:sip_disintegrate}
\end{equation}
$ \forall A \in \mathcal{B}_{\mathcal{Z}}$ is a consistent solution to the stochastic inverse problem in the sense of~\eqref{eq:consistency} and is uniquely determined for a given prior density $\pi_{\mathrm{init}}$ on $\left(\mathcal{Z}, \mathcal{B}_{\mathcal{Z}}\right)$ where its push-forward defines a predicted density, $\pi_{\mathrm{pred}}$, satisfying the predictability assumption that $\exists$ $C>0$ such that $\pi_{\mathrm{pred}}(y)\leq C\pi_{\mathrm{obs}}(y)$ $\forall y\in\mathcal{D}$.
\end{theorem}

We generally work with the densities directly rather than the integral in~\eqref{eq:sip_disintegrate} as this is sufficient for generating samples. 
To this end, we formally define the solution to the stochastic inverse problem as an update to the initial density via the formula
\begin{equation}
    \pi_{\mathrm{up}} \left(\mathbf{z}_{0}\right) = \pi_{\mathrm{init}} \left(\mathbf{z}_{0}\right) \frac{\pi_{\mathrm{obs}}\left( Q\left(\mathbf{z}_{0}\right) \right) }{\pi_{\mathrm{pred}} \left( Q \left(\mathbf{z}_{0}\right) \right)}. \label{eq:dci_problem2}
\end{equation}

\subsection{Predictability Assumption}
The predictability assumption of Theorem~\ref{thm:sip_exist_unique} is actually a statement about the properties of both the initial density and the QoI map. 
This assumption implies $\supp(\pi_{\mathrm{obs}}) \subseteq \supp(\pi_{\mathrm{pred}})$, ensuring that observable values with positive likelihood are also predictable, which is fundamental to solvability. 
It is therefore critical that we are able to assess whether or not the predictability assumption is satisfied. 

We develop a numerical diagnostic to verify if the predictability assumption is met by first recognizing that if it is satisfied then $\pi_{\mathrm{up}}$ is in fact a density from which it immediately follows that:
\begin{equation}\label{eq:e_r}
    \mathbb{E}\left(r\left(Q \left(\mathbf{z}_{0}\right)\right)\right) = \int_{\mathcal{Z}} r\left(\mathbf{z}_{0}\right) \pi_{\mathrm{init}}\left(\mathbf{z}_{0}\right) \mu_{\mathcal{Z}} = \int_{\mathcal{Z}} \pi_{\mathrm{up}} \left(\mathbf{z}_{0} \right) \mu_{\mathcal{Z}} = 1,
\end{equation}
where
\begin{equation}\label{eq:r}
r(Q(\mathbf{z}_{0})):=\frac{\pi_{\mathrm{obs}}(Q(\mathbf{z}_{0}))}{\pi_{\mathrm{pred}}(Q(\mathbf{z}_{0}))}.
\end{equation}
From this, we see that any Monte Carlo based estimate of the sample mean of $r(Q(\mathbf{z}_0))$ should be approximately unity if the predictability assumption holds. 
This result is widely referenced as a numerical diagnostic and applied throughout much of the DCI literature (e.g., see~\cite{Butler2018, Pilosov2023, del-Castillo-Negrete2024}).
In particular, a significant departure of the sample mean of $\mathbb{E}(r)$ from unity beyond what one might expect from finite sample error indicates potential violations of the predictability assumption.

\subsection{MUD points}\label{app:MUD}

In the context of this work, we assume a finite amount of (possibly noisy) data on a QoI map is obtained for an unknown state $\mathbf{z}_{0}^{\dagger}$, and we estimate this state with the maximal updated density (MUD) point defined by
\begin{equation}
 z^{\mathrm{MUD}}:= \underset{\mathbf{z}_{0}\in Z}{\arg\max} \ \pi_{\mathrm{up}}(\mathbf{z}_{0}).
\end{equation}
We refer the interested reader to the work of \cite{Pilosov2023} that first considered the problem of point estimation within the DCI context using MUD points as well as provided the full linear Gaussian theory and a comparison to Bayesian MAP points and least-squares estimates.
The work of \cite{del-Castillo-Negrete2024} extended this further to sequential MUD point estimation that included change-point-detection for time-shifting parameters. 

%\\\\\\\\\\\\\\\\\\\\\\\\\\\\\\\\\\\\\\\\\\\\\\\\\\\\\\\\\\\\\\\\\\\\\\\\\\\\\\\\\\\\\\\\\\\\\\\\\\\\\\\\\\\\\\\\\\\\\\\\\\\\\\\\\\\\\\\\\\\\\\\\\\\\\\\\\\\\\\\\\\\\\\\\\\\

\section{Proofs of DCI 4D-Var Results}
We now establish the existence and uniqueness of the minimizer of the DC 4D-Var function $ \mathcal{J}_{\mathrm{DC}} $, commonly referred to as the MUD point, and we note that the extension of these results to DC-WME 4D-Var function follow naturally through appropriate substitution of terms.
While prior work has proven related uniqueness results in parameter estimation (e.g.,~\cite{Pilosov2023}), we provide a complete and rigorous proof tailored to the state estimation setting. 
The argument proceeds by establishing four main components: strict convexity, coercivity, existence, and uniqueness.  
In the theorems below, we adopt the following notation and assumptions. 
First, we assume that $ \mathcal{J}_{\mathrm{DC}} $ denotes the data-consistent 4D-Var cost function as defined in~\eqref{eq:dc_4dvar_cost}. 
The control variable, denoted $ \mathbf{z}_0 $, belongs to the state space $ \mathcal{Z} $, which we assume to be a reflexive Banach space. 
Additionally, we assume the predictability assumption from Theorem~\ref{thm:sip_exist_unique} is satisfied. 
For all time indices $ k $, the norms $ \| \cdot \|_{\mathbf{B}^{-1}_k} $, $ \| \cdot \|_{\mathbf{R}^{-1}_k} $, and $ \| \cdot \|_{\mathbf{L}^{-1}_k} $ are induced by the symmetric positive-definite operators $ \mathbf{B}_k $, $ \mathbf{R}_k $, and $ \mathbf{L}_k $, respectively.

\begin{assumption}[Quantity of Interest Map]
We make the following assumptions for the QoI map used in the data-consistent 4D-Var function.
\begin{enumerate}
\item Each map $ Q_k $ is full rank. \label{ass:qoi_map_1}
\item Each map $ Q_k $ is Fréchet differentiable. \label{ass:qoi_map_2}
\item For each $k$, $ Q_k : \mathcal{Z} \to \mathcal{D}_k $ be a bounded linear operator into a Hilbert space $ \mathcal{D}_k $\label{ass:qoi_map_3}
\item $ Q_k : \mathcal{Z} \to \mathcal{D} $ are weak-to-strong continuous QoI maps, i.e. whenever a sequence $\mathbf{z}_n \rightharpoonup \mathbf{z}$ weakly in $\mathcal{Z}$, it follows that $Q_k\left(\mathbf{z}_n\right) \rightarrow Q_k(\mathbf{z})$ strongly in $\mathcal{D}$.\label{ass:qoi_map_4}
\item Each map $ Q_k $ satisfies a linear growth condition: there exists $ C_k > 0 $ such that
\begin{equation}\label{ass:qoi_map_5}
\| Q_k(\mathbf{z}_0) \|_{\mathcal{D}} \leq C_k (1 + \| \mathbf{z}_0 \|_{\mathcal{Z}}), \quad \forall \mathbf{z}_0 \in \mathcal{Z}.
\end{equation}
\end{enumerate}
\end{assumption}

%\\\\\\\\\\\\\\\\\\\\\\\\\\\\\\\\\\\\\\\\\\\\\\\\\\\\\\\\\\\\\\\\\\\\\\\\\\\\\\\\\\\\\\\\\\\\\\\\\\\\\\\\\\\\\\\\\\\\\\\\\\\\\\\\\\\\\\\\\\\\\\\\\\\\\\\\\\\\\\\\\\\\\\\\\\\
\subsection{Convexity}\label{append:convexity}

We first show that $ \mathcal{J}_{\mathrm{DC}} $ is strictly convex on the reflexive Banach space $ \mathcal{Z} $. 
At a high-level, the result follows from the structure of the function, which comprises a sum of quadratic terms involving bounded linear QoI maps $ Q_k: \mathcal{Z} \rightarrow \mathcal{D}_k $, each assumed to be full rank. We combine this with the positive definiteness of the background, observation, and predictability precision operators, and invoke the predictability assumption to show that the resulting Hessian defines a strictly positive definite bilinear form

\begin{lemma}[Strict Convexity of $\mathcal{J}_{\mathrm{DC}}$]\label{lemma:convexity}
Under Assumptions~\ref{ass:qoi_map_1} and~\ref{ass:qoi_map_3} $\mathcal{J}_{\mathrm{DC}}$ is strictly convex on $\mathcal{Z}$
\end{lemma}
\begin{proof}
The function $ \mathcal{J}_{\mathrm{DC}} $ consists of quadratic and affine terms in $ \mathbf{z}_0 $, which ensures twice Fréchet differentiability. At any point $ \mathbf{z}_0 \in \mathcal{Z} $, the Hessian of $ \mathcal{J}_{\mathrm{DC}} $ defines a symmetric bilinear form $ D^2 \mathcal{J}_{\mathrm{DC}}(\mathbf{z}_0) : \mathcal{Z} \times \mathcal{Z} \to \mathbb{R} $, given by
\begin{equation}
\begin{split}
D^2 \mathcal{J}_{\mathrm{DC}}(\mathbf{z}_0)\!\left(\mathbf{v}, \mathbf{w}\right) 
    &= \langle \mathbf{B}^{-1} \mathbf{v}, \mathbf{w} \rangle + \sum_{k=0}^N \langle \mathbf{R}_k^{-1} Q_k \mathbf{v}, \, Q_k \mathbf{w} \rangle \\
    &\quad - \sum_{k=0}^N \langle \mathbf{L}_k^{-1} Q_k \mathbf{v}, \, Q_k \mathbf{w} \rangle.
\end{split}
\end{equation}

Introduce the ratio precision operator:
\begin{equation}
\mathbf{W}_{k}^{-1} := \mathbf{R}_k^{-1} - \mathbf{L}_k^{-1}.
\end{equation}
Substituting into the expression for the Hessian yields
\begin{equation}
D^2 \mathcal{J}_{\mathrm{DC}}(\mathbf{z}_0)\left(\mathbf{v}, \mathbf{v}\right) =
\langle \mathbf{B}^{-1} \mathbf{v}, \mathbf{v} \rangle
+ \sum_{k=0}^N \langle \mathbf{W}_{k}^{-1} Q_k \mathbf{v}, Q_k \mathbf{v} \rangle.
\end{equation}
To establish strict convexity, it suffices to verify coercivity of the bilinear form. That is, a constant $ c > 0 $ must exist such that
\begin{equation}
D^2 \mathcal{J}_{\mathrm{DC}}(\mathbf{z}_0)\left(\mathbf{v}, \mathbf{v}\right) \geq c \| \mathbf{v} \|_{\mathcal{Z}}^2 \quad \text{for all } \mathbf{v} \in \mathcal{Z}.
\end{equation}
Since $ \mathbf{B}^{-1} $ is symmetric and positive definite, the Rayleigh quotient guarantees the bound
\begin{equation}
\langle \mathbf{B}^{-1} \mathbf{v}, \mathbf{v} \rangle \geq \lambda_{\min}\left(\mathbf{B}^{-1}\right) \| \mathbf{v} \|_{\mathcal{Z}}^2
\end{equation}
for some constant $ \lambda_{\min}\left(\mathbf{B}^{-1}\right) > 0 $. Next, observe that each term $ \langle \mathbf{W}_{k}^{-1} Q_k \mathbf{v}, Q_k \mathbf{v} \rangle $ is nonnegative and becomes strictly positive when $ Q_k \mathbf{v} \neq 0 $, due to the assumption $ \mathbf{W}_{k}^{-1} \succ 0 $ (i.e., the predictability assumption). Full-rank property of each operator $ Q_k $ implies that the condition $ Q_k \mathbf{v} = 0 $ for all $ k $ leads to $ \mathbf{v} = 0 $. Consequently, the sum contributes strictly positively whenever $ \mathbf{v} \neq 0 $, so
\begin{equation}
D^2 \mathcal{J}_{\mathrm{DC}}(\mathbf{z}_0)\left(\mathbf{v}, \mathbf{v}\right) > 0 \quad \text{for all } \mathbf{v} \in \mathcal{Z} \setminus \{0\}.
\end{equation}
This result confirms that the second derivative defines a strictly positive definite bilinear form. Therefore, $ \mathcal{J}_{\mathrm{DC}} $ is strictly convex.
\end{proof}
%\\\\\\\\\\\\\\\\\\\\\\\\\\\\\\\\\\\\\\\\\\\\\\\\\\\\\\\\\\\\\\\\\\\\\\\\\\\\\\\\\\\\\\\\\\\\\\\\\\\\\\\\\\\\\\\\\\\\\\\\\\\\\\\\\\\\\\\\\\\\\\\\\\\\\\\\\\\\\\\\\\\\\\\\\\\

\subsection{Coercivity}

While we established coercivity of the Hessian of $\mathcal{J}_{\mathrm{DC}}$ in the prior proof, we next verify that $ \mathcal{J}_{\mathrm{DC}} $ is itself coercive, which allows us to apply the direct method in the calculus of variations. 
Specifically, we demonstrate that $ \mathcal{J}_{\mathrm{DC}}(\mathbf{z}_0) \to \infty $ as $ \| \mathbf{z}_0 \|_{\mathcal{Z}} \to \infty $. 
This behavior arises from the dominant quadratic growth of the background term and a linear growth condition on each $ Q_k $, which ensures that the negative predictability penalty remains subdominant. 
A coercive function confines any minimizing sequence $ \{ \mathbf{z}_m \} \subset \mathcal{Z} $ to a sequentially compact subset of $ \mathcal{Z} $, a consequence of reflexivity. 

\begin{lemma}[Coercivity of the DC 4D-Var function]\label{lemma:coercivity}
Under Assumption~\ref{ass:qoi_map_5} $\mathcal{J}_{\mathrm{DC}}$ is coercive on $\mathcal{Z}$;that is,
\begin{equation}
\| \mathbf{z}_0 \|_{\mathcal{Z}} \to \infty \quad \Rightarrow \quad \mathcal{J}_{\mathrm{DC}}(\mathbf{z}_0) \to \infty.
\end{equation}
\end{lemma}
\begin{proof}
Define $ \mathbf{v} := \mathbf{z}_0 - \mathbf{z}_0^{\mathrm{b}} \in \mathcal{Z} $. Boundedness and coercivity of $ \mathbf{B}^{-1} $ imply the existence of a constant $ \lambda_{\min}\left(\mathbf{B}^{-1}\right) > 0 $ such that
\begin{equation}
\| \mathbf{v} \|_{\mathbf{B}}^2 \geq \lambda_{\min}\left(\mathbf{B}^{-1}\right) \| \mathbf{v} \|_{\mathcal{Z}}^2.
\end{equation}
The background term admits the lower bound
\begin{equation}
\frac{1}{2} \| \mathbf{z}_0 - \mathbf{z}_0^{\mathrm{b}} \|_{\mathbf{B}}^2 \geq \frac{\lambda_{\min}\left(\mathbf{B}^{-1}\right)}{2} \| \mathbf{z}_0 \|_{\mathcal{Z}}^2 - C,
\end{equation}
where the constant $ C > 0 $ depends only on $ \left\| \mathbf{z}_0^{\mathrm{b}} \right\| $. This inequality follows from the reverse triangle inequality. The data misfit term satisfies the nonnegativity property, since each $ Q_k(\mathbf{z}_0) \in \mathcal{D} $:
\begin{equation}
\frac{1}{2} \sum_{k=0}^N \| \mathbf{y}_k - Q_k(\mathbf{z}_0) \|_{\mathbf{R}^{-1}_k}^2 \geq 0.
\end{equation}
The predictability penalty term admits the upper bound obtained from the linear growth assumption:
\begin{equation}
\left\|Q_k\left(\mathbf{z}_0\right)-Q_k\left(\mathbf{z}_0^b\right)\right\|_{\mathcal{D}} \leq C_k\left(2+\left\|\mathbf{z}_0\right\|_{\mathcal{Z}}+\left\|\mathbf{z}_0^b\right\|_{\mathcal{Z}}\right).
\end{equation}
A quadratic polynomial bound in $ \left\| \mathbf{z}_0 \right\|_{\mathcal{Z}}^2 $ then follows by squaring both sides:
\begin{equation}
\left\|Q_k\left(\mathbf{z}_0\right)-Q_k\left(\mathbf{z}_0^b\right)\right\|_{\mathcal{D}}^2 \leq C_k^{\prime}\left(1+\left\|\mathbf{z}_0\right\|_{\mathcal{Z}}^2\right),     \label{eq:lin_growth2}
\end{equation}
where the constant $ C_k^{\prime} > 0 $ absorbs the cross terms and the fixed background norm $ \left\| \mathbf{z}_0^b \right\| $.
Applying~\eqref{eq:lin_growth2} to the weighted predictability norm yields
\begin{equation}
\left\| Q_k(\mathbf{z}_0) - Q_k(\mathbf{z}_0^{\mathrm{b}}) \right\|_{\mathbf{L}^{-1}_k}^2 \leq \lambda_{\max }\left(\mathbf{L}_k^{-1}\right)C^{\prime}_k(1 + \| \mathbf{z}_0 \|_{\mathcal{Z}}^2).
\end{equation}
The entire third term is then bounded above as
\begin{equation}
\left| \frac{1}{2} \sum_{k=0}^N \| Q_k(\mathbf{z}_0) - Q_k(\mathbf{z}_0^{\mathrm{b}}) \|_{\mathbf{L}^{-1}_k}^2 \right| \leq C^{*} (1 + \| \mathbf{z}_0 \|_{\mathcal{Z}}^2),
\end{equation}
for some constant $ C^{*} > 0 $. Combining all contributions gives
\begin{equation}
\begin{split}
\mathcal{J}_{\mathrm{DC}}(\mathbf{z}_0) 
    &\geq \frac{\lambda_{\min}\!\left(\mathbf{B}^{-1}\right)}{2} 
        \| \mathbf{z}_0 \|_{\mathcal{Z}}^2 
        - C^{*}\!\left(1 + \| \mathbf{z}_0 \|_{\mathcal{Z}}^2\right) \\
    &= \left( \frac{\lambda_{\min}\!\left(\mathbf{B}^{-1}\right)}{2} - C^{*} \right) 
        \| \mathbf{z}_0 \|_{\mathcal{Z}}^2 - C^{*}.
\end{split}
\end{equation}
Since $ \lambda_{\min}\left(\mathbf{B}^{-1}\right) > 0 $ and $ C^{*} $ is fixed, the right-hand side grows without bound as $ \| \mathbf{z}_0 \|_{\mathcal{Z}} \to \infty $, provided $ \frac{\lambda_{\min}\left(\mathbf{B}^{-1}\right)}{2} > C^{*} $ or under asymptotic growth. Therefore, coercivity of $ \mathcal{J}_{\mathrm{DC}} $ follows:
\begin{equation}
\lim_{\| \mathbf{z}_0 \|_{\mathcal{Z}} \to \infty} \mathcal{J}_{\mathrm{DC}}(\mathbf{z}_0) = \infty.
\end{equation}
\end{proof}
%\\\\\\\\\\\\\\\\\\\\\\\\\\\\\\\\\\\\\\\\\\\\\\\\\\\\\\\\\\\\\\\\\\\\\\\\\\\\\\\\\\\\\\\\\\\\\\\\\\\\\\\\\\\\\\\\\\\\\\\\\\\\\\\\\\\\\\\\\\\\\\\\\\\\\\\\\\\\\\\\\\\\\\\\\\\
\subsection{Existence}

Having established coercivity, we now prove the existence of a minimizer. 
Reflexivity ensures that every bounded minimizing sequence admits a weakly convergent subsequence.
If $ \mathcal{J}_{\mathrm{DC}} $ is weakly lower semicontinuous, then it attains its infimum at the weak limit of such a sequence.
We confirm this lower semicontinuity by appealing to the weak-to-strong continuity of the maps $ Q_k $, which guarantees that the norm-squared terms in the observation and predictability penalties vary continuously under weak convergence. 

\begin{theorem}[Existence of Minimizer for DC 4D-Var]
Under Assumptions~\ref{ass:qoi_map_4}, \ref{ass:qoi_map_5}, and Lemma~\ref{lemma:coercivity}, $\mathcal{J}_{\mathrm{DC}}$ is weakly lower semicontinuous and coercive on $ \mathcal{Z} $, and thus admits a minimizer.  \label{thm:existance}
\end{theorem}
\begin{proof}
Consider a sequence $ \{ \mathbf{z}_m \} \subset \mathcal{Z} $ that converges strongly to some $ \mathbf{z} \in \mathcal{Z} $. Continuity of each map $ Q_k $, along with continuity of norms in Hilbert spaces, allows the limit to pass through each term:
\begin{equation}
\lim_{m \to \infty} \mathcal{J}_{\mathrm{DC}}(\mathbf{z}_m) = \mathcal{J}_{\mathrm{DC}}(\mathbf{z}).
\end{equation}
This identity establishes strong continuity of $ \mathcal{J}_{\mathrm{DC}} $, which implies weak lower semicontinuity:
\begin{equation}
\mathcal{J}_{\mathrm{DC}}(\mathbf{z}) \leq \liminf_{m \to \infty} \mathcal{J}_{\mathrm{DC}}(\mathbf{z}_m).
\end{equation}
Now assume that a sequence $ \{ \mathbf{z}_k \} \subset \mathcal{Z} $ converges weakly to some $ \mathbf{z}_0 \in \mathcal{Z} $. 
By definition of the $\liminf$, there exists a subsequence $ \{ \mathbf{z}_{k_j} \} $ such that
\begin{equation}
\lim_{k \to \infty} \mathcal{J}_{\mathrm{DC}}(\mathbf{z}_k) = \lim_{j \to \infty} \mathcal{J}_{\mathrm{DC}}(\mathbf{z}_{k_j}).
\end{equation}
Reflexivity of $ \mathcal{Z} $ permits application of the Banach–Saks Theorem~\cite{Brezis2011}, which yields a further subsequence $ \{ \mathbf{z}_i \} $ whose Cesàro mean
\begin{equation}
\tilde{\mathbf{z}}_n := \frac{1}{n} \sum_{i=1}^n \mathbf{z}_i
\end{equation}
converges strongly to $ \mathbf{z}_0 $ in $ \mathcal{Z} $. 
Since $ \mathcal{J}_{\mathrm{DC}} $ is convex by Lemma~\ref{lemma:convexity}, Jensen's inequality implies
\begin{equation}
\frac{1}{n} \sum_{i=1}^n \mathcal{J}_{\mathrm{DC}}(\mathbf{z}_i) \geq \mathcal{J}_{\mathrm{DC}}(\tilde{\mathbf{z}}_n),
\end{equation}
and strong convergence $ \tilde{\mathbf{z}}_n \to \mathbf{z}_0 $ ensures
\begin{equation}
\lim_{n \to \infty} \mathcal{J}_{\mathrm{DC}}(\tilde{\mathbf{z}}_n) \geq \mathcal{J}_{\mathrm{DC}}(\mathbf{z}_0).
\end{equation}
It follows that
\begin{equation}
\lim_{k \to \infty} \mathcal{J}_{\mathrm{DC}}(\mathbf{z}_k) = \lim_{n \to \infty} \frac{1}{n} \sum_{i=1}^n \mathcal{J}_{\mathrm{DC}}(\mathbf{z}_i) \geq \mathcal{J}_{\mathrm{DC}}(\mathbf{z}_0),
\end{equation}
verifying that $ \mathcal{J}_{\mathrm{DC}} $ is weakly lower semicontinuous. Coercivity and weak lower semicontinuity of $ \mathcal{J}_{\mathrm{DC}} $, together with reflexivity of $ \mathcal{Z} $, satisfy the conditions of the direct method in the calculus of variations. Existence of a minimizer therefore follows. That is, there exists $ \mathbf{z}^* \in \mathcal{Z} $ such that
\begin{equation}
\mathcal{J}_{\mathrm{DC}}(\mathbf{z}^*) = \inf_{\mathbf{z}_0 \in \mathcal{Z}} \mathcal{J}_{\mathrm{DC}}(\mathbf{z}_0).
\end{equation}
\end{proof}
\noindent In the Data Consistent Inversion literature we refer to the mimimizer $\mathbf{z}^*$ as the Maximal Updated Density point $\mathbf{z}^{\mathrm{MUD}}$.
%\\\\\\\\\\\\\\\\\\\\\\\\\\\\\\\\\\\\\\\\\\\\\\\\\\\\\\\\\\\\\\\\\\\\\\\\\\\\\\\\\\\\\\\\\\\\\\\\\\\\\\\\\\\\\\\\\\\\\\\\\\\\\\\\\\\\\\\\\\\\\\\\\\\\\\\\\\\\\\\\\\\\\\\\\\\
\subsection{Uniqueness}

Finally, the strict convexity of $ \mathcal{J}_{\mathrm{DC}} $ ensures uniqueness. 
If two distinct minimizers existed, they would contradict the strict convexity inequality, which only permits a single minimizer.
We therefore conclude that $ \mathcal{J}_{\mathrm{DC}} $ admits a unique global minimizer $ \mathbf{z}^{\mathrm{MUD}} \in \mathcal{Z} $, and that this minimizer satisfies the first-order optimality condition.

\begin{theorem}[Uniqueness of the DC 4D-Var Minimizer]
Under Assumption~\ref{ass:qoi_map_2} and Lemma~\ref{lemma:convexity} $\mathcal{J}_{\mathrm{DC}}$  admits a unique minimizer $ \mathbf{z}^{\text{MUD}} \in \mathcal{Z} $.    \label{thm:uniqueness}
\end{theorem}
\begin{proof}
Existence of a minimizer $ \mathbf{z}^{\text{MUD}} \in \mathcal{Z} $ ensures that
\begin{equation}
\mathcal{J}_{\mathrm{DC}}(\mathbf{z}^{\text{MUD}}) = \inf_{\mathbf{z}_0 \in \mathcal{Z}} \mathcal{J}_{\mathrm{DC}}(\mathbf{z}_0).
\end{equation}
Strict convexity of $ \mathcal{J}_{\mathrm{DC}} $ permits application of the strict convexity inequality: for all $ \bar{\mathbf{z}}, \mathbf{z} \in \mathcal{Z} $ with $ \bar{\mathbf{z}} \neq \mathbf{z} $,
\begin{equation}
\mathcal{J}_{\mathrm{DC}}(\bar{\mathbf{z}}) > \mathcal{J}_{\mathrm{DC}}(\mathbf{z}) + D\mathcal{J}_{\mathrm{DC}}(\mathbf{z})\left(\bar{\mathbf{z}} - \mathbf{z}\right),
\end{equation}
where $ D\mathcal{J}_{\mathrm{DC}}(\mathbf{z}) \in \mathcal{Z}^* $ denotes the Fréchet derivative of $ \mathcal{J}_{\mathrm{DC}} $ at $ \mathbf{z} $.
The condition $ D\mathcal{J}_{\mathrm{DC}}(\mathbf{z}^{\text{MUD}}) = 0 $ implies that $ \mathbf{z}^{\text{MUD}} $ satisfies a first-order optimality condition and defines a strict local minimizer.
Uniqueness follows by contradiction. Suppose two distinct minimizers $ \mathbf{z}_1 \neq \mathbf{z}_2 $ exist. Then strict convexity implies
\begin{equation}
\mathcal{J}_{\mathrm{DC}} \left( \frac{1}{2}(\mathbf{z}_1 + \mathbf{z}_2) \right) < \frac{1}{2} \mathcal{J}_{\mathrm{DC}}(\mathbf{z}_1) + \frac{1}{2} \mathcal{J}_{\mathrm{DC}}(\mathbf{z}_2).
\end{equation}
This inequality contradicts the assumption that both $ \mathbf{z}_1 $ and $ \mathbf{z}_2 $ minimize $ \mathcal{J}_{\mathrm{DC}} $, since their average would attain a strictly lower value. Therefore, strict convexity guarantees that $ \mathcal{J}_{\mathrm{DC}} $ admits a unique minimizer $ \mathbf{z}^{\text{MUD}} \in \mathcal{Z} $, which satisfies the first-order condition
\begin{equation}
D\mathcal{J}_{\mathrm{DC}}(\mathbf{z}^{\text{MUD}}) = 0.
\end{equation}
\end{proof}

%\\\\\\\\\\\\\\\\\\\\\\\\\\\\\\\\\\\\\\\\\\\\\\\\\\\\\\\\\\\\\\\\\\\\\\\\\\\\\\\\\\\\\\\\\\\\\\\\\\\\\\\\\\\\\\\\\\\\\\\\\\\\\\\\\\\\\\\\\\\\\\\\\\\\\\\\\\\\\\\\\\\\\\\\\\\
%\\\\\\\\\\\\\\\\\\\\\\\\\\\\\\\\\\\\\\\\\\\\\\\\\\\\\\\\\\\\\\\\\\\\\\\\\\\\\\\\\\\\\\\\\\\\\\\\\\\\\\\\\\\\\\\\\\\\\\\\\\\\\\\\\\\\\\\\\\\\\\\\\\\\\\\\\\\\\\\\\\\\\\\\\\\

\newpage
\section{Adjoint DC-4DVar Algorithms}\label{append:adjoint_algos}
\begin{figure}[htbp]
\centering
\begin{minipage}[t]{0.48\textwidth}
\begin{algorithm}[H]
\caption{DC 4D-Var}
\begin{algorithmic}[1]
\State \textbf{Input:} $ \mathbf{z}_0^{\mathrm{b}}, \{\mathbf{y}_k\}_{k=0}^N , Q_k,\mathcal{H}_k, \mathcal{M}_k, \mathbf{B}, \mathbf{L}_k, \mathbf{R}_{k}$
\State \textbf{Output:} Gradient $ \nabla \mathcal{J}_{\mathrm{DC}}(\mathbf{z}_0) $
\Statex
\Statex \textbf{\large{\underline{Forwards Pass:}}}
\Statex
\Statex \textbf{Forward Model Propagation:}
\State $\mathbf{z}_0 \gets$ initial control guess
\For{$k = 0$ to $N-1$}
    \State $\mathbf{z}_{k+1} \gets \mathcal{M}_k(\mathbf{z}_k)$
\EndFor
\Statex
\Statex \textbf{\large{\underline{Backwards Pass:}}}
\Statex
\Statex \textbf{Compute Terminal Adjoint:}
\State $ \mathbf{r}_{N} \gets Q_N - \mathbf{y}_N$
\State $ \mathbf{q}_{N} \gets Q_N - Q_N^{\mathrm{b}}$
\State $\boldsymbol{\lambda}_N \gets (\nabla Q_N)^\top \left[\mathbf{R}_N^{-1} \mathbf{r}_N - \mathbf{L}_N^{-1} \mathbf{q}_{N}\right]$
\Statex
\Statex \textbf{Adjoint Recursion:}
\For{$k = N-1$ to $0$}
    \State $ \mathbf{r}_{k} \gets Q_k - \mathbf{y}_k$
    \State $ \mathbf{q}_{k} \gets Q_k - Q_k^{\mathrm{b}}$
    \State $\boldsymbol{\lambda}_k \gets (\nabla Q_k)^\top \left[\mathbf{R}_k^{-1} \mathbf{r}_{k} - \mathbf{L}_k^{-1} \mathbf{q}_{k}\right] + (\nabla \mathcal{M}_k)^\top \boldsymbol{\lambda}_{k+1}$
\EndFor
\Statex
\Statex \textbf{Gradient Evaluation:}
\State $\nabla \mathcal{J}_{\mathrm{DC}}(\mathbf{z}_0) \gets \mathbf{B}_0^{-1} (\mathbf{z}_0 - \mathbf{z}_0^{\mathrm{b}}) - \boldsymbol{\lambda}_0$
\end{algorithmic}
\end{algorithm}
\end{minipage}%
\hfill
\begin{minipage}[t]{0.48\textwidth}
\begin{algorithm}[H]
\caption{DC-WME 4D-Var}
\begin{algorithmic}[1]
\State \textbf{Input:} $ \mathbf{z}_0^{\mathrm{b}}, \{\mathbf{y}_k\}_{k=0}^N , \mathcal{H}_k, \mathcal{M}_k, \mathbf{B}, \mathbf{L}_k, \mathbf{R}_{k}$
\State \textbf{Output:} Gradient $ \nabla \mathcal{J}_{\mathrm{DC}}(\mathbf{z}_0) $
\Statex
\Statex \textbf{\large{\underline{Forwards Pass:}}}
\Statex
\Statex \textbf{Forward Model Propagation:}
\State $\mathbf{z}_0 \gets$ initial control guess
\For{$k = 0$ to $N-1$}
    \State $\mathbf{z}_{k+1} \gets \mathcal{M}_k(\mathbf{z}_k)$
\EndFor
\Statex
\Statex \textbf{\large{\underline{Backwards Pass:}}}
\Statex
\Statex \textbf{Compute WME:}
\State $Q_{\mathrm{wme}}(\mathbf{z}_0) \gets \frac{1}{\sqrt{N}} \sum\limits_{k=1}^N \mathbf{R}_{k}^{-1/2} \left(\mathcal{H}_k \mathbf{z}_k - \mathbf{y}_k\right)$ 
\State $\mathbf{q}_{\mathrm{wme}} = Q_{\mathrm{wme}}(\mathbf{z}_0) - Q_{\mathrm{wme}}(\mathbf{z}_0^{\mathrm{b}})$
\State $\mathbf{s} \gets Q_{\mathrm{wme}}\left(\mathbf{z}_0\right) - \mathbf{L}_{\mathrm{wme}}^{-1}\mathbf{q}_{\mathrm{wme}}$
\Statex
\Statex \textbf{Compute Terminal Adjoint:}
\State $\mathbf{J}_N \gets \frac{1}{\sqrt{N}}  \mathbf{R}_{N}^{-\frac{1}{2}}\mathcal{H}_N$
\State $\boldsymbol{\lambda}_{N} \gets \mathbf{J}_N^\top \mathbf{s}$
\Statex
\Statex \textbf{Adjoint Recursion:}
\For{$k = N-1$ to $0$}
    \State $\mathbf{J}_k \gets \frac{1}{\sqrt{N}}  \mathbf{R}_{k}^{-\frac{1}{2}}\mathcal{H}_k$
    \State $\nabla_{\mathbf{z}_k} \mathcal{J}_k \gets \mathbf{J}_k^\top \mathbf{s}$
    \State $\boldsymbol{\lambda}_k \gets D_{\mathbf{z}_k} \mathcal{M}_k^\top \boldsymbol{\lambda}_{k+1} + \nabla_{\mathbf{z}_k} \mathcal{J}_k$
\EndFor
\Statex
\Statex \textbf{Gradient Evaluation:}
\State $\nabla \mathcal{J}_{\mathrm{DC}}(\mathbf{z}_0) \gets \mathbf{B}_0^{-1} (\mathbf{z}_0 - \mathbf{z}_0^{\mathrm{b}}) + \boldsymbol{\lambda}_0$
\end{algorithmic}
\end{algorithm}
\end{minipage}
\end{figure}

 \bibliographystyle{elsarticle-num} 
 \bibliography{references}

\end{document}